\documentclass[11pt,a4paper]{article}
\usepackage[utf8]{inputenc}

\usepackage[all]{xy}
\usepackage{latexsym}
\usepackage[margin=3cm]{geometry}
\usepackage[dvipdfmx]{graphicx}

\usepackage{hyperref}
\usepackage{mymacro}
\usepackage{tikz-cd}
\usepackage{mathtools}
\usepackage{physics}

\usepackage[
backend=biber,
style=ext-alphabetic,
maxalphanames = 5,
maxcitenames  = 3,
minalphanames = 3,
minnames      = 1,
sorting=nyt,
giveninits,
articlein=false,
url=false,  
doi=false,
eprint=true,
isbn=false
]{biblatex}
\renewcommand{\labelenumi}{$\mathrm{(\roman{enumi})}$}
\renewcommand{\labelenumii}{$\mathrm{(\alph{enumii})}$}

\SelectTips{cm}{}

\makeatletter

\@addtoreset{equation}{section}
\makeatother

\title{Lagrangian cobordism and shadow distance \\ in Tamarkin category}
\author{Tomohiro Asano \and Yuichi Ike \and Wenyuan Li}
\date{\today}

\begin{document}
\maketitle

\begin{abstract}
We study exact Lagrangian cobordisms between exact Lagrangians in a cotangent bundle in the sense of Arnol'd, using microlocal theory of sheaves. We construct a sheaf quantization for an exact Lagrangian cobordism between Lagrangians with conical ends, prove an iterated cone decomposition of the sheaf quantization for cobordisms with multiple ends, and show that the interleaving distance of sheaves is bounded by the shadow distance of the cobordism. Using the result, we prove a rigidity result on Lagrangian intersection by estimating the energy cost of splitting and connecting Lagrangians through cobordisms. 
\end{abstract}

\section{Introduction}\label{section:introduction}
Lagrangian cobordisms between Lagrangian submanifolds were first introduced by Arnol'd \cite{Arnold80I,Arnold80II}. Recently, Biran--Cornea \cite{BC13,BC14} discovered the deep relation between Lagrangian cobordisms and Fukaya categories, and Cornea--Shelukhin \cite{CS19} introduced the shadow distance of Lagrangian cobordisms, which was then used by Biran--Cornea--Shelukhin \cite{BCS18} to show rigidity results of Lagrangian intersections. 

On the other hand, the microlocal theory of sheaves developed by Kashiwara--Schapira \cite{KS90} has been shown to be a powerful tool in symplectic geometry following the pioneering work of Nadler--Zaslow \cite{NZ,Nad} and Tamarkin \cite{Tamarkin}. 
Sheaf-theoretic methods have been effectively applied to symplectic geometry, for example,  \cite{GKS,Chiu17,Gu23,Shende19conormal,GPS24microlocal,zhang2024capacities}, and have advantage in the studies of for example degenerate intersections of Lagrangians \cite{Ike19,AISQ} and non-smooth limits of Lagrangians \cite{AI24completeness,GV24coisotropic}.

In this paper, we apply sheaf theory to the study of Lagrangian cobordisms, explore the algebraic relations of sheaves induced by Lagrangian cobordisms, and compare the interleaving distance introduced by the first two authors \cite{AI20} with the shadow distance of Lagrangian cobordisms \cite{CS19}, which leads to rigidity results of Lagrangian intersections.

\subsection{Main results}

Let $M$ be a closed manifold, $T^*M$ be its cotangent bundle and $T^{*,\infty}M$ be its cosphere bundle.
Let $L_0,\dots,L_{m-1}$ and $L'_0,\dots,L'_{n-1} \subset T^*M$ be exact Lagrangian submanifolds with conical ends. 
An exact Lagrangian cobordism with conical ends $V \subset T^*(M \times \bR)$ between $(L_0,\dots,L_{m-1})$ and $(L'_0,\dots,L'_{n-1})$ is an exact Lagrangian with conical end $V_\infty \subset T^{*,\infty}(M \times \bR)$ such that $V_\infty \cap T^{*,\infty}(M \times [-1, 1])$ is compact and
\begin{align*}
    V \cap T^*(M \times (-\infty, -1]) & = \bigcup_{i=0}^{m-1}L_i \times (-\infty, -1] \times \{i\}, \\ 
    V \cap T^*(M \times [+1, +\infty)) & = \bigcup_{j=0}^{n-1}L'_j \times [+1, +\infty) \times \{j\}.
\end{align*}

In this study, we use the Tamarkin category $\cT(T^*M)$, which is a localization of the sheaf category $\Sh(M \times \bR)$. 
We can equip $\cT(T^*M)$ with an interleaving distance $d'_{\cT(T^*M)}$ that is stable under Hamiltonian isotopies \cite{AI20} (see \cref{subsec:Tamarkin} for details). 
For objects in $\cT(T^*M)$, their reduced microsupports are subsets in $T^*M$.
For any exact Lagrangian $L \subset T^*M$ with conical ends, we can construct a simple sheaf $F_L \in \cT(T^*M)$ with reduced microsupport on $L$ \cite{Gu12,JT_brane,Vi19,Gu23}.
They are called sheaf quantizations of $L$.

Our first result is the following existence theorem of a sheaf quantization of a Lagrangian cobordism, namely, that there exists a simple sheaf with reduced microsupport in the given exact Lagrangian cobordism.

\begin{theorem}[{see \cref{thm:sheaf-quan-cob-filling}}]
    Let $V \subset T^*(M \times \bR)$ be an exact Lagrangian cobordism with conical ends between $(L_0,\dots,L_{m-1})$ and $(L'_0,\dots,L'_{n-1})$ with vanishing Maslov class and relative second Stiefel--Whitney class. 
    Then there is a simple sheaf quantization $F \in \cT(T^*(M \times \bR))$.
\end{theorem}

\begin{remark}[see \cref{rem:quan-jintreumann}]
    Jin--Treumann \cite[Subsection 1.11.4]{JT_brane} have mentioned that such a theorem should be true after deforming the Lagrangian cobordism so that it becomes lower exact and satisfies \cite[Assumption 3.7]{JT_brane}, which requires an extra proof as $V$ is not eventually conical in $T^*(M \times \bR)$. Here, we obtain such a result via a different approach without using \cite[Assumption 3.7]{JT_brane}.
\end{remark}

Our second result is the following iterated cone decomposition theorem of a sheaf quantization of a Lagrangian cobordism in the Tamarkin category $\cT(T^*M)$. We write
\begin{multline*}
    [ G_0 \to G_1 \to \dots \to G_{k-1} \to G_k ] \\
    \coloneqq \Cone(\Cone( \cdots \Cone(G_0[-k] \to G_1[1-k]) \to \dots \to G_{k-1}[-1]) \to G_k).
\end{multline*}

\begin{theorem}[{see \cref{theorem:iterated_cone}}]\label{theorem:iterated_cone_intro}
    Let $F$ be a simple sheaf quantization of Lagrangian cobordism $V$ between $(L_0,\dots,L_{m-1})$ and $(L'_0,\dots,L'_{n-1})$.
    Set $F_s \coloneqq F|_{M \times \bR_t \times \{s\}}$. 
    Then one has iterated decompositions in $\cT(T^*M)$
    \begin{align*}
        F_{-1} & =
        \ld F_{L_0} \to F_{L_1} \to \dots \to F_{L_{m-2}} \to F_{L_{m-1}} \rd, \\ 
        F_{+1} & =
        \ld F_{L'_{n-1}} \to F_{L'_{n-2}} \to \dots \to F_{L'_1} \to F_{L'_0} \rd,
    \end{align*}   
    where $F_{L_i}$ (resp.\ $F_{L'_j}$) is a simple sheaf quantization of $L_i$ (resp.\ $L'_j)$.    
\end{theorem}

Our third result is the following inequality for the interleaving distance by the shadow distance of a Lagrangian cobordism.
For a Lagrangian cobordism $V \subset T^*(M \times \bR)$, the shadow $\cS(V)$ of $V$ is defined as the area of $T^*\bR \setminus \cU(V)$, where $\cU(V)$ is the union of unbounded regions of $T^*\bR \setminus p(V)$ with $p \colon T^*(M \times \bR) \to T^*\bR$ being the projection map. 
It is known that in the exact setting, the Lagrangian shadow induces a metric on all Lagrangian submanifolds \cite{CS19}.
Our result in particular implies an isomorphism of the sheaves in the Tamarkin category modulo torsion objects $\cT_\infty(T^*M) = \cT(T^*M) / \mathrm{Tor}$ where sheaves related by Hamiltonian isotopies become isomorphic objects (see \cref{remark:category_T} for the precise definition). 

\begin{theorem}[{see \cref{theorem:inequality_multiple_ends}}]\label{theorem:inequality_intro}
    Let $F$ be a simple sheaf quantization of Lagrangian cobordism $V$ and consider iterated cone decompositions as in \cref{theorem:iterated_cone_intro}.
    Then there exist increasing sequences of real numbers $\{c_i\}_{i=1,\dots,m-1}, \{c'_j\}_{j=0,\dots,n-1}$ such that 
    \begin{equation*}
        d'_{\cT(T^*M)}(\tilde{F}_{-1},\tilde{F}_{+1}) \le \cS(V),
    \end{equation*}
    where 
    \begin{align*}
        \tilde{F}_{-1} & \coloneqq 
        \ld F_{L_0} \to T_{c_1}F_{L_1} \to \dots \to T_{c_{m-2}}F_{L_{m-2}} \to T_{c_{m-1}}F_{L_{m-1}} \rd, \\ 
        \tilde{F}_{+1} & \coloneqq 
        \ld T_{-c'_{n-1}} F_{L'_{n-1}} \to T_{-c'_{n-2}}F_{L'_{n-2}} \to \dots \to T_{-c'_1}F_{L'_1} \to T_{-c'_{0}}F_{L'_{0}} \rd. 
    \end{align*}   
    In particular, one has $\tilde{F}_{-1} \simeq \tilde{F}_{+1}$ in $\cT_\infty(T^*M)$.
\end{theorem}

We do not need to require $V$ to be a smooth embedded Lagrangian cobordism. In fact, we prove \cref{theorem:iterated_cone_intro,theorem:inequality_intro} only assuming that the projection of the reduced microsupport of $F$ onto $T^*\bR$ satisfies some condition (see \cref{assumption:proj} more precisely).

We apply the theorems above to Lagrangian intersection, recovering \cite[Thm.~C]{BCS18} in the cotangent bundle case, which shows that there is an energy cost to separate or connect Lagrangians through cobordisms. 

\begin{theorem}[{see \cref{theorem:intersection} for a more precise statement}]\label{theorem:intersection_intro}
    Let $V \subset T^*(M \times \bR)$ be an exact Lagrangian cobordism with conical ends between $(L_0, \dots, L_{m-1})$ and $L'$. 
    Moreover, let $N$ be a Lagrangian and assume that the Lagrangians are in general position in the sense that each pair of $N,L_0,\dots,L_{m-1}$, and $L'$ intersect transversally and there are no point in the intersection of any triple $(N, L_i, L_j)$ for $i \neq j$. 
    Then there exists a constant $\delta>0$ that depends only on $L_0, \dots, L_{m-1}$, and $N$ satisfying the following: 
    if $\cS(V) < \delta$ then 
    \begin{equation*}
        \#(N \cap L') \ge \sum_{i=0}^{m-1} \#(N \cap L_i).
    \end{equation*}
\end{theorem}

\begin{remark}
    In fact, \cref{theorem:intersection} is a more general statement assuming only clean intersections of Lagrangian submanifolds, which means that $T_p(L_0 \cap L_1) = T_pL_0 \cap T_pL_1$ for any $p \in L_0 \cap L_1$. 
    See Sect.~\ref{subsec:lagrangian-intersection} for the argument. Such situations arise naturally in examples.
    When $L_0$ and $L_1$ intersect cleanly along $D$, one can define the surgery $L_0 \#_D L_1$ along $D$ and build a Lagrangian cobordism from $L_0, L_1$ to $L_0 \#_D L_1$ \cite[Lem.~6.1]{MakWu}. 
\end{remark}

Compared to previous Lagrangian intersection results using sheaf theoretic methods \cite{Ike19,AISQ}, we need to estimate the energy of higher homotopies in the iterated cone decomposition, which, in Lagrangian Floer theory, come from pseudo-holomorphic polygons instead of just pseudo-holomorphic strips. See the main body of the paper.

\subsection{Related work}
The relation between Lagrangian cobordisms and Fukaya categories and their filtrations has been studied in many previous works. Biran--Cornea \cite{BC13,BC14} showed that (monotone) Lagrangian cobordisms induce iterated cone decompositions in the Fukaya categories. Cornea--Shelukhin \cite{CS19} introduced the shadow distance of Lagrangian cobordisms, and later Biran--Cornea--Shelukhin \cite{BCS18} used the shadow distance to prove results on Lagrangian intersection using the theory of weakly filtered $A_\infty$-categories. More recently, motivated by the geometry, Biran--Cornea--Zhang \cite{BCZ23} developed the theory of triangulated persistence categories and showed that the Tamarkin category and the Fukaya category are triangulated persistence categories.

It is well known that the Hofer distance of $C^1$-close Lagrangians is related to the shadow of Lagrangian movie under Hamiltonian isotopies \cite{Milinkovic}. Since its introduction, the Lagrangian shadow has become a standard tool in the study of quantitative aspects of Lagrangian submanifolds. See \cite{Hicks23,HicksMak22} for interesting geometric constructions that involve the Lagrangian shadows.

The algebraic results above can also be enhanced to a categorical level. Biran--Cornea \cite{BC14} defined a category that with a functor into the cone resolution category of the Fukaya category, while Nadler--Tanaka \cite{NadTanaka} defined a stable $\infty$-category of Lagrangian cobordisms with an exact functor into the wrapped Fukaya category \cite{Tanaka16,Tanaka16Exact}. We expect that the sheaf quantization result here provides a way to define a functor from that category to the sheaf categories.

Finally, it is expected that the relation between Lagrangian cobordisms and Fukaya categories passes through the Waldhausen $s$-construction in algebraic $K$-theory. In particular, Biran--Cornea \cite{BC14} conjectured a relation between the Lagrangian cobordism group and the Grothendieck group of the Fukaya category \cite{BC17,Bosshard23}. 
See \cite{BCZ23-2} for the relation between the Grothendieck group and the persistence of the Fukaya category. It would be interesting to study the question for the sheaf categories.

\section{Preliminaries}

In this section, we recall some basics of the microlocal sheaf theory due to Kashiwara and Schapira~\cite{KS90}.
Throughout this paper, all manifolds are assumed to be of class $C^\infty$ without boundary.
Let $\bfk$ be a field until the end of this paper.

Until the end of this section, let $X$ be a $C^\infty$-manifold.
For a locally closed subset $A$ of $X$, we denote by $\overline{A}$ its closure and by $\Int(A)$ its interior.
We also denote by $\Delta_X$ or simply $\Delta$ the diagonal of $X \times X$.
We denote by $\pi_X \colon T^*X \to X$ the cotangent bundle of $X$.
We write $0_X$ for the zero-section of $T^*X$ and set $\rT X \coloneqq T^*X \setminus 0_X$.
We denote by $(x;\xi)$ a local homogeneous coordinate system of $T^*X$.
The cotangent bundle $T^*X$ is an exact symplectic manifold with the Liouville 1-form $\lambda_{T^*X}=\langle \xi, dx \rangle$.

\subsection{Microsupports of sheaves}\label{subsec:microsupport}

For a manifold $X$, we write $\bfk_X$ for the constant sheaf with stalk $\bfk$ and let $\Sh(X)=\Sh(\bfk_X)$ denote the dg-derived category of sheaves of $\bfk$-vector spaces on $X$. 
One can define Grothendieck's six operations between dg-derived categories of sheaves $\cHom,\allowbreak \otimes, \allowbreak f_*,\allowbreak f^{-1},\allowbreak f_!,\allowbreak f^!$ for a morphism of manifolds $f \colon X \to Y$ \cite{Spaltenstein,Schn18}. 
Note that all the functors in this paper are dg-derived functors.
For a locally closed subset $Z$ of $X$, we denote by $\bfk_Z \in \Sh(X)$ the constant sheaf with stalk $\bfk$ on $Z$, extended by $0$ on $X \setminus Z$.
Moreover, for a locally closed subset $Z$ of $X$ and $F \in \Sh(X)$, we define 
\begin{equation*}
    F_Z \coloneqq F \otimes \bfk_Z, \quad \Gamma_Z(F) \coloneqq \cHom(\bfk_Z,F).
\end{equation*}
Let $X_1, X_2, X_3$ be manifolds and $\pi_{ij} \colon X_1 \times X_2 \times X_3 \to X_i \times X_j$ for $i \neq j \in \{1,2,3\}$. 
For $K_{12} \in \Sh(X_1 \times X_2)$ and $K_{23} \in \Sh(X_2 \times X_3)$, their composition is defined by 
\begin{equation*}
    K_{12} \circ K_{23} = \pi_{13 !}(\pi_{12}^{-1} K_{12} \otimes \pi_{23}^{-1} K_{23}).
\end{equation*}

Following \cite{KS90,robalo2018lemma}, for $F \in \Sh(X)$, we let $\MS(F)$ denote the microsupport of $F$, which is a conical (i.e., invariant under the action of $\bR_{>0}$) closed subset of $T^*X$.
The set $\MS(F)$ describes the singular codirections of $F$, in which $F$ does not propagate.
We also use the notation $\mathring{\MS}(F) \coloneqq \MS(F) \cap \rT X=\MS(F) \setminus 0_X$.

By using microsupports, we can microlocalize the category $\Sh(X)$.
Let $A \subset T^*X$ be a subset and set $\Omega=T^*X \setminus A$.
We denote by $\Sh_{A}(X)$ the full dg-subcategory of $\Sh(X)$ consisting of sheaves whose microsupports are contained in $A$.
By the triangle inequality, the subcategory $\Sh_{A}(X)$ is a pretriangulated subcategory.
We set
\begin{equation*}
    \Sh(X;\Omega) \coloneqq \Sh(X)/\Sh_{A}(X),
\end{equation*}
the dg-quotient of $\Sh(X)$ by $\Sh_A(X)$.
A morphism $u \colon F \to G$ in $\Sh(X)$ becomes a quasi-isomorphism in $\Sh(X;\Omega)$ if $u$ is embedded in an exact triangle $F \overset{u}{\to} G \to H \overset{+1}{\to}$ with $\MS(H) \cap \Omega=\varnothing$.
We also denote by $\Sh_{(A)}(X)$ the full dg-subcategory of $\Sh(X)$ formed by $F$ for which there exists a neighborhood $U$ of $A$ satisfying $\MS(F) \cap U \subset A$.

\subsection{Tamarkin category and interleaving distance}\label{subsec:Tamarkin}

Here, we recall the Tamarkin category $\cT(T^*X)$ and the interleaving distance. 

We write $(x;\xi)$ for a local homogeneous coordinate system on $T^*X$ and $(t;\tau)$ for the homogeneous coordinate system on $T^*\bR_t$.
Define the maps
\begin{equation*}
\begin{aligned}
    & \qquad \pi_1, \pi_2, m \colon X \times \bR \times \bR \to X \times \bR_t, \\
    \pi_1(x,t_1,t_2)&=(x,t_1), \ \pi_2(x,t_1,t_2)=(x,t_2), \ m(x,t_1,t_2)=(x,t_1+t_2).
\end{aligned}
\end{equation*}
For $F,G \in \Sh(X \times \bR_t)$, one sets
\begin{align}
    F \star G & \coloneqq m_!(\pi_1^{-1}F \otimes \pi_2^{-1}G), \label{eq:star} \\
    \cHom^\star(F,G) & \coloneqq \pi_{1*} \cHom(\pi_2^{-1}F,m^!G). \nonumber
\end{align}
Note that the functor $\star$ is a left adjoint to $\cHom^\star$.

Tamarkin proved that the localized category $\Sh(X \times \bR_t;\{\tau >0\})$ is equivalent to both the left orthogonal ${}^\perp \Sh_{\{\tau \le 0\}}(X \times \bR_t)$ and the right orthogonal $\Sh_{\{\tau \le 0\}}(X \times \bR_t)^\perp$:
\begin{equation*}
\begin{aligned}
    P_l \coloneqq \bfk_{X \times [0,+\infty)} \star (\ast) \colon \Sh(X \times \bR_t;\{\tau >0\}) \simto {}^\perp \Sh_{\{\tau \le 0\}}(X \times \bR_t), \\
    P_r \coloneqq \cHom^\star(\bfk_{X \times [0,+\infty)}, \ast) \colon \Sh(X \times \bR_t;\{\tau >0\}) \simto \Sh_{\{\tau \le 0\}}(X \times \bR_t)^\perp.
\end{aligned}
\end{equation*}
We set $\Omega_+ \coloneqq \{\tau >0\} \subset T^*(X \times \bR_t)$ and define a map $\rho \colon \Omega_+ \to T^*X$ as $(x,t;\xi,\tau) \mapsto (x;\xi/\tau)$.

\begin{definition}
	One defines 
	\begin{equation*}
	\cT(T^*X) \coloneqq \Sh(X \times \bR_t;\Omega_+) \simeq {}^\perp \Sh_{\{\tau \le 0\}}(X \times \bR_t) \simeq \Sh_{\{\tau \le 0\}}(X \times \bR_t)^\perp.
	\end{equation*}
    For $F \in \cT(T^*X)$, one defines its reduced microsupport by 
    \begin{equation*}
        \RS(F) \coloneqq \overline{\rho(\MS(F) \cap \Omega_+)} \subset T^*X.
    \end{equation*}
    For a closed subset $A$ of $T^*X$, one also defines a full subcategory $\cT_A(T^*X)$ of $\cT(T^*X)$ by 
    \begin{equation*}
        \cT_A(T^*X) \coloneqq \{ F \in \cT(T^*X) \mid \RS(F) \subset A \}.
    \end{equation*}
\end{definition}

For $F \in \cT(T^*X)$, we take the canonical representative $P_l(F) \in {}^\perp \Sh_{\{\tau \le 0\}}(X \times \bR_t)$ unless otherwise specified.
Note that if $F \in {}^\perp \Sh_{\{\tau \le 0\}}(X \times \bR_t)$, then $\cHom^\star(F,G) \in \Sh_{\{\tau \le 0\}}(X \times \bR_t)^\perp$.
Thus $\cHom^\star$ induces an internal Hom functor $\cHom^\star \colon \cT(T^*X)^{\mathrm{op}} \times \cT(T^*X) \to \cT(T^*X)$.
We also often regard $\cHom^\star(F,G)$ as in ${}^\perp \Sh_{\{\tau \le 0\}}(X \times \bR_t)$ by applying the projector $P_l$.

\begin{proposition}[{cf.~\cite[Lem.~4.18]{GS14}}]\label{prp:morD}
    For $F,G \in \cT(T^*X)$, one has
    \begin{equation*}
       \Hom_{\cT(T^*X)}(F,G) \simeq
       \Gamma_{X \times [0,+\infty)}(X \times \bR_t;\cHom^\star(F,G)). 
    \end{equation*}
\end{proposition}

Now we introduce a pseudo-distance $d'_{\cT(T^*X)}$ on the Tamarkin category $\cT(T^*X)$.
Note that our distance $d'_{\cT(T^*X)}$ here is different from the distances used in \cite{AI20, AISQ,AI24completeness}. 
Here we use weakly $(a, b)$-isomorphic to obtain an isomorphism in $\cT_\infty(T^*X)$ (see \cref{remark:category_T} below), while in \cite{AI20, AISQ} the authors only use $(a, b)$-interleaving in the definition. 
Nevertheless, all the results in \cite{AI20,AISQ} hold for this stronger notion \cite[Remark 4.5]{AI20}. 
For $c \in \bR$, consider the translation map $T_c \colon X \times \bR \to X \times \bR, \, T_c(x, t) = (x, t+c)$. 
By abuse of notation, we write $T_c \coloneqq {T_c}_* \colon \Sh(X \times \bR_t) \to \Sh(X \times \bR_t)$ and the induced functor ${T_c}_* \colon \cT(T^*X) \to \cT(T^*X)$. 
For $c \le d$, there exists a natural transformation $\tau_{c,d} \colon T_c \to T_d$ between the functors on $\cT(T^*X)$.

\begin{definition}\label{def:defab}
	Let $F,G \in \cT(T^*X)$ and $a,b \in \bR_{\ge 0}$.
	\begin{enumerate}
		\item The pair $(F,G)$ is said to be said to be \emph{weakly $(a,b)$-isomorphic} 
		if there exist morphisms $\alpha, \delta \colon F \to {T_a}G$ 
		and $\beta, \gamma \colon G \to {T_b}F$ such that
		\begin{enumerate}
			\renewcommand{\labelenumii}{$\mathrm{(\arabic{enumii})}$}
			\item $F \xrightarrow{\alpha} {T_a} G \xrightarrow{{T_a}\beta} {T_{a+b}}F$ is homotopic to $\tau_{0,a+b}(F) \colon F \to {T_{a+b}}F$,
			\item $G \xrightarrow{\gamma} {T_b} F \xrightarrow{{T_b}\delta} {T_{a+b}}G$ is homotopic to $\tau_{0,a+b}(G) \colon G \to {T_{a+b}}G$, and
			\item $\tau_{a,2a}(G) \circ \alpha \simeq \tau_{a,2a}(G) \circ \delta$ and $\tau_{b,2b}(F) \circ \beta \simeq \tau_{b,2b}(F) \circ \gamma$.
		\end{enumerate}
		\item  One defines 
    	\begin{equation*}
    	d'_{\cT(T^*X)}(F,G)
    	\coloneqq
    	\inf 
    	\lc 
    	a+b \in \bR_{\ge 0} 
    	\relmid
        \begin{aligned}
            & a,b \in \bR_{\ge 0}, \\
    	    & \text{$(F,G)$ is weakly $(a,b)$-isomorphic}
        \end{aligned}
    	\rc.
    	\end{equation*}
	\end{enumerate}	 
\end{definition}

\begin{remark}\label{remark:category_T}
    We can define a new category $\cT_\infty(T^*X)$ as the quotient category \linebreak $\cT(T^*X)/\Tor$, where $\Tor$ is the null system consisting of objects whose distances with the zero object are finite.
    Then it is shown in \cite{GS14} that
    \begin{equation*}
        \Hom_{\cT_\infty(T^*X)}(F,G) \coloneqq \hocolim_{c \to \infty} \Hom_{\cT(T^*X)}(F,T_c G).
    \end{equation*}
    Hence, if $(F,G)$ are weakly $(a,b)$-isomorphic for some $a,b$, then the images of $F$ and $G$ are isomorphic in $\cT_\infty(T^*X)$. 
\end{remark}

We will use the following proposition in the proof of the main theorem.

\begin{proposition}[{\cite[Prop.~4.15]{AI20} and \cite[Prop.~3.9]{AI24completeness}}]\label{prp:htpy_wabisom}
	Let $\cH \in \Sh_{\{\tau \ge 0\}}(X \times \bR \times I)$.
	Assume that there exist continuous functions $f, g \colon I \to \bR_{\ge 0}$ satisfying
	\begin{equation*}
	\MS(\cH) \subset T^*X \times \{(t,s;\tau,\sigma) \mid -f(s) \cdot \tau \le \sigma \le g(s) \cdot \tau \}.
	\end{equation*}
	Then $\left(\cH|_{X \times \bR \times \{0\}},\cH|_{X \times \bR \times \{1\}} \right)$ is weakly $\left( \int_{0}^{1} g(s) ds+\varepsilon, \int_{0}^{1} f(s) ds +\varepsilon \right)$-isomorphic for any $\varepsilon \in \bR_{>0}$.
	In particular, one has $d'_{\cT(T^*X)}\left(\cH|_{X \times \bR \times \{0\}},\cH|_{X \times \bR \times \{1\}} \right) \le \int_{0}^{1} (f+g)(s) ds$.
\end{proposition}

We note that in the above proposition, the condition that both $f$ and $g$ are non-negative is essential and cannot be removed.

\subsection{Sheaf quantization of Hamiltonian isotopy}

Here, we recall sheaf quantization of Hamiltonian isotopies due to Guillermou--Kashiwara--Schapira~\cite{GKS}.

We first recall the homogeneous case.
Let $I$ be an open interval containing the closed interval $[0,1]$ and $Y$ be a manifold.
For a homogeneous Hamiltonian isotopy $\wh{\phi} \colon \rT Y \times I \to \rT Y$, we can define a conical Lagrangian submanifold $\Lambda_{\wh{\phi}} \subset \rT Y \times \rT Y \times T^*I$ whose projection onto $\rT Y \times \rT Y \times I$ is the graph of~$\widehat{\phi}$.

\begin{theorem}[{\cite[Thm.~3.7]{GKS}}]\label{thm:GKS}
    Let $\wh{\phi} \colon \rT Y \times I \to \rT Y$ be a homogeneous Hamiltonian isotopy.
    There exists a unique object $K \in \Sh(Y \times Y \times I)$ satisfying the following conditions:
	\begin{enumerate}
		\renewcommand{\labelenumi}{$\mathrm{(\arabic{enumi})}$}
		\item $\mathring{\MS}(K) \subset \Lambda_{\wh{\phi}}$,
		\item $K|_{Y \times Y \times \{0\}} \simeq \bfk_{\Delta_{Y}}$.
	\end{enumerate}
	Moreover, both projections $\Supp(K) \to Y \times I$ are proper. 
\end{theorem}

The object $K$ is called the \emph{sheaf quantization} or the \emph{GKS kernel} of $\wh{\phi}$.

Now we consider the non-homogeneous case. 
Let $\phi^H=(\phi^H_z)_{z \in I} \colon T^*X \times I \to T^*X$ be a Hamiltonian isotopy associated with a Hamiltonian function $H \colon T^*X \times I \to \bR$.
Here $z$ denotes the time parameter. 
We assume that there exists a subset $C$ of $T^*X$ such that $\pi_X$ is proper on $C$ and $\Supp(H) \subset C \times I$. 
Note that the Hamiltonian vector field is defined by $d\lambda_{T^*X}(X_{H_z},\ast)=-dH_{z}$ and $\phi^H$ is the identity for $z=0$.
One can conify $\phi^H$ and construct $\wh{\phi}$ such that $\wh{\phi}$ lifts $\phi^H$ as follows.
Define $\wh{H} \colon T^*X \times \rT \bR \times I \to \bR$ by  $\wh{H}_z(x,t;\xi,\tau) \coloneqq \tau \cdot H_z(x;\xi/\tau)$, which is homogeneous of degree $1$.
By the condition, we can extend $\wh{H}$ to $\rT(X \times \bR_t) \times I$.
We write $\wh{\phi} \colon \rT(X \times \bR) \times I \to \rT(X \times \bR)$ for the homogeneous Hamiltonian isotopy associated with $\wh{H}$, such that
\begin{equation*}
    \wh{\phi}_z(x,t;\xi,\tau)
    =
    (x',t+u_z(x;\xi/\tau);\xi',\tau),
\end{equation*}
where $(x';\xi'/\tau)=\phi^H_z(x;\xi/\tau)$ and 
\begin{equation}\label{eqn:hamilton-lift}
    u_z(x; \xi/\tau) = \int_0^z (H_{z'} - \lambda_{T^*X}(X_{H_{z'}})) \circ \phi_{z'}^H(x; \xi/\tau)\, dz'.
\end{equation}
See \cite[Subsection~A.3]{GKS} for more details.

We can apply \cref{thm:GKS} in the case where $Y=X \times \bR_t$ and $\wh{\phi}$ is constructed from $\phi^H \colon T^*X \times I \to T^*X$ as above to get the GKS kernel $K \in \Sh(X \times \bR \times X \times \bR \times I)$.
Set $K_{z} \coloneqq K|_{X \times \bR \times X \times \bR \times \{z\}} \in \Sh(X \times \bR \times X \times \bR)$.
It is also proved by Guillermou--Schapira~\cite[Prop.~4.29]{GS14} that the composition with $K_z$ defines a functor
\begin{equation*}
    K_z \circ (\ast) \colon \cT(T^*X) \to \cT(T^*X).
\end{equation*}
Moreover, for $F \in \cT_A(T^*X)$ and any $z \in I$, we have $K_z \circ F \simeq (K \circ F)|_{X \times \{z\}} \in \cT_{\phi^H_{z}(A)}(T^*X)$.
In other words, the composition $K_z \circ (\ast)$ induces a functor $\cT_{A}(T^*X) \to \cT_{\phi^H_s(A)}(T^*X)$ for any closed subset $A$ on $T^*X$.

\begin{remark}\label{rem:noncompact-gks}
    If $K \in \Sh(X_2 \times \bR \times X_2 \times \bR \times I)$ is the sheaf quantization of a Hamiltonian isotopy $\phi \colon T^*X_2 \times I \to T^*X_2$, the object $\bfk_{\Delta_{X_1}} \boxtimes K \in \Sh(X_1 \times X_2 \times \bR \times X_1 \times X_2 \times \bR \times I)$ defines a sheaf quantization of $\id_{T^*X_1} \times \phi \colon T^*(X_1 \times X_2) \times I \to T^*(X_1 \times X_2)$. This allows us to deal with products of properly supported Hamiltonians, which are usually non-proper.
    We will use this fact in \cref{section:shadow}.
\end{remark}

\section{Sheaf quantization of Lagrangian cobordism}

In what follows, until the end of the paper, let $M$ be a connected manifold without boundary.
In this section, we prove the existence of a sheaf quantization of exact Lagrangian cobordisms $V \subset T^*(M \times \bR_s)$, namely a simple sheaf $F_V \in \cT_V(T^*(M \times \bR_s))$.

Let $L \subset T^*M$ be a closed exact Lagrangian submanifold such that the restriction of the Liouville form $\lambda_{T^*M}|_L = df_L$. Define the conification of $L$ to be
\begin{equation*}
    \Lambda_L = \{(x, t; \xi, \tau) \mid t = f(x, \xi/\tau), \, \tau > 0\} \subset T^*_{\tau > 0}(M \times \bR_t).
\end{equation*}
The quotient $\Lambda_L/\bR_{>0} \subset T^*_{\tau > 0}(M \times \bR_t) / \bR_{>0}$ defines the Legendrian lift of $L$. Guillermou \cite{Gu12} proved that there exists a simple sheaf $F \in \Sh_{\Lambda_L}(M \times \bR_t)$. 
This was generalized by the third author to certain non-compact exact Lagrangians \cite[Section~3.3]{Li23}. We will follow that construction.

\subsection{Kashiwara-Schapira stack}
First, we introduce the Kashiwara--Schapira stack \cite[Section~6]{Gu12} or \cite[Part~X]{Gu23}, which is also called the brane category in \cite{JT_brane}. For a conical Lagrangian $\Lambda \subset T^*M$, the obstruction for the existence of a global section in the Kashiwara--Schapira stack is going to be purely algebraic topological.

\begin{definition}\label{def:KSstack}
    Let $\Lambda \subset T^*M$ be a locally closed conical Lagrangian.
    \begin{enumerate}
        \item The Kashiwara--Schapira prestack $\muSh^\text{pre}_\Lambda$ is a presheaf of dg-categories on $\Lambda$, where for any open $\Lambda_0 \subset \Lambda$, the objects are $\Sh_{(\Lambda_0)}(M)$, and morphisms are
        \begin{equation*}
            \Hom_{\muSh^\text{pre}_\Lambda}(F, G) = \Hom_{\Sh(M; \Lambda_0)}(F, G);
        \end{equation*}
        \item The \emph{Kashiwara--Schapira stack} $\muSh_\Lambda$ is the stack associated to the prestack $\muSh^\text{pre}_\Lambda$.
    \end{enumerate}
\end{definition}

\begin{theorem}[Guillermou {\cite[Thm.~11.5~and~12.5]{Gu12}}]\label{thm:KSstack}
    Let $\Lambda \subset T^*M$ be a locally closed conical Lagrangian. 
    Suppose $\Lambda$ has vanishing Maslov class and vanishing relative second Stiefel--Whitney class. Then there exists a simple object in the Kashiwara--Schapira stack $\cF \in \muSh_\Lambda(\Lambda)$, and they are classified by rank~1 local systems on $\Lambda$.
\end{theorem}

\subsection{Guillermou doubling functor}
For a conical Lagrangian $\Lambda \subset T^*_{\tau > 0}(M \times \bR_t)$, we can define 
\begin{align*}
    \Lambda_q &= \{(x, t, u; \xi, \tau, 0) \in T^*(M \times \bR_t \times \bR_u) \mid (x, t; \xi, \tau) \in \Lambda, \tau > 0\}, \\
    \Lambda_r &= \{(x, t, u; \xi, \tau, -\tau) \in T^*(M \times \bR_t \times \bR_u) \mid (x, t-u; \xi, \tau) \in \Lambda, \tau > 0\}.
\end{align*}
When $\Lambda / \bR_{>0}$ is compact, Guillermou \cite[Section 13-15]{Gu12} defined a doubling functor from the Kashiwara--Schapira stack $\muSh_\Lambda(\Lambda)$ to a simple sheaf $F_\mathrm{dbl}$, whose microsupport 
\begin{equation*}
    \mathring{\MS}(F_\mathrm{dbl}) \subset (\Lambda_q \cup \Lambda_r) \cap T^*(M \times \bR_t \times (0, \varepsilon)).
\end{equation*}
In the previous work of the third author \cite[Section 3.3]{Li23}, we generalized the theorem to the case where $\Lambda / \bR_{>0}$ is non-compact under a suitable assumption, that there exists a tubular neighborhood of the Lagrangian submanifold of positive radius with respect to some Riemannian metric adapted to the symplectic form.

Let $X$ be a manifold with a Riemannian metric $g$ and $L \subset X$ be an embedded submanifold. A (tubular) neighborhood $U$ of $L$ of positive radius $r > 0$ with respect to a metric $g$ on $X$ is a (tubular) neighborhood $U$ such that for any $x \in X$ with $d_g(x, L) \leq r$, we have $x \in U$.

The Riemannian metric we are going to fix is the following. Let $g$ be a complete Riemannian metric on $M$. It induces a complete Riemannian metric $g_{T^*M}$ on $T^*M$. We call $g_{T^*M}$ a complete adapted metric on the symplectic manifold $T^*M$, following the notion of \cite[Section 2.2.6]{EliashbergGromov98}.
We set $T^{*,\infty}M \coloneqq \rT M/\bR_{>0}$, which we identify with $\{|\xi|=1 \}$.

\begin{theorem}[{\cite[Cor.~3.16]{Li23}}]\label{thm:doubling-largeshift}
    Let $L \subset T^*M$ be an exact Lagrangian submanifold with a properly embedded conification $\Lambda_L \subset T^*(M \times \bR_t)$ whose Maslov class and relative second Stiefel--Whitney class vanish. 
    If $L \subset T^*M$ admits a tubular neighborhood of positive radius with respect to the complete adapted metric, then for any $u > 0$, there exists a doubled sheaf
    \begin{equation*}
        F_{\mathrm{dbl}} \in \Sh_{\Lambda_q \cup \Lambda_r}(M \times \mathbb{R}_t \times (0, u))
    \end{equation*}
    that is simple along $\Lambda_q \cap T^*(M \times \bR_t \times (0, u))$.
\end{theorem}

\subsection{Sheaf quantization of Lagrangian cobordisms}

\cref{thm:doubling-largeshift} enables us to work with sheaf quantization problems for more general non-compact Lagrangian submanifolds. 
For instance, we may consider asymptotically conical Lagrangian cobordisms between asymptotic Lagrangian submanifolds, which is already dealt with in \cite{JT_brane}. 

\begin{definition}
    Let $L \subset T^*M$ be an exact Lagrangian. One says that $L$ is an \emph{exact Lagrangian with conical end} if $L \cap T^*M$ is a conical Lagrangian outside a neighborhood $U \subset T^*M$ of the zero section such that $\pi|_U \colon U \rightarrow M$ is proper.
    One sets $L_\infty \coloneqq (L \setminus U)/\bR_{>0} \subset T^{*,\infty}M$.
\end{definition}
\begin{remark}
    In particular, $L_\infty$ is a Legendrian submanifold of $T^{*,\infty}M$. 
    Such a Lagrangian $L$ is called a \emph{Lagrangian filling} of the Legendrian $L_\infty$.
\end{remark}

Here is a technical prerequisite. For exact Lagrangians in $T^*M$ with closed Legendrian boundary in $T^{*,\infty}M$, by considering the standard adapted metric on $T^*M$, the following lemma is almost immediate:

\begin{lemma}\label{lem:nbhd-filling}
    Let $L \subset T^{*}M $ be a Lagrangian filling with compact Legendrian boundary $\Lambda \subset T^{*,\infty}M$. Then $L$ has a tubular neighborhood of positive radius $r > 0$ with respect to the adapted metric on $T^*M$.
\end{lemma}
\begin{proof}
    We only need to find a tubular neighborhood of a positive radius outside a compact neighborhood of the zero section. Suppose $L \cap \{ |\xi|>r \} = \Lambda \times (r, +\infty)$. Observe that the standard adapted metric $g_{T^*M}$ can be written as $s^2g_{T^{*,\infty}M} + ds^2$. When $r$ is large, it is bounded from below by the product metric $g_{T^{*,\infty}M} + ds^2$. Since the compact Legendrian $\Lambda$ has a tubular neighborhood of positive radius, we can conclude that so does $L$.
\end{proof}

In order to show existence of sheaf quantization for the non-compact Lagrangian $V \subset T^*(M \times \bR_s)$, a necessary condition is that the conification $\Lambda_V/\bR_{>0} \subset \rT(M \times \bR_s \times \bR_t) / \bR_{>0}$ is properly embedded. 
Therefore, we need to carefully choose the primitive $f_V$ with $\lambda_{T^*(M \times \bR_s)}|_V = df_V$.

Following Jin--Treumann~\cite{JT_brane}, we consider the following condition. 
The proof of Jin--Treumann works also in this slightly generalized setting since we have assumed that the Lagrangian is conical outside a subset that projects properly onto $M$.
    
\begin{lemma}[{\cite[Subsection~3.6]{JT_brane}}]
    Let $L \subset T^*M$ be an exact Lagrangian with conical ends. 
    There exists a Hamiltonian isotopy from $V$ to $V'$ such that for any $R > 0$, the primitive $f_{V'}$ with $\lambda_{T^*(M \times \bR_s)}|_{V'} = df_{V'}$ is bounded from below on $V' \cap T^*(M \times [-R, R])$.
\end{lemma}

Now we introduce the notion of exact Lagrangian cobordisms with conical ends.

\begin{definition}\label{def:lag-cob}
    Let $L_0,\dots,L_{m-1}$ and $L'_0,\dots,L'_{n-1} \subset T^*M$ be exact Lagrangian submanifolds with conical ends (both of which can be empty). 
    An \emph{exact Lagrangian cobordism with conical ends} $V \subset T^*(M \times \bR_s)$ between $(L_0,\dots, L_{m-1})$ and $(L'_0,\dots,L'_{n-1})$ is a properly embedded exact Lagrangian submanifold with conical end $V_\infty$ such that $V_\infty \cap T^{*,\infty}(M \times [-1, 1])$ is compact and
    \begin{align*}
        V \cap T^*(M \times (-\infty, -1]) & = \bigcup_{i=0}^{m-1}L_i \times (-\infty, -1] \times \{i\}, \\ 
        V \cap T^*(M \times [+1, +\infty)) & = \bigcup_{j=0}^{n-1}L'_j \times [+1, +\infty) \times \{j\}.
    \end{align*}
\end{definition}

\begin{remark}
    When we allow more than one connected components on each ends, then the connected sum of two exact Lagrangian cobordisms may not be exact anymore, unless we include the choice of primitives of $L_i$'s and $L'_j$'s as part of the data for the Lagrangian brane and require the primitive of $V$ to restrict to the chosen primitives of $L_i$'s and $L'_j$'s. 
\end{remark}

Our main result in this section is the existence of sheaf quantization for exact Lagrangian cobordisms. Recall that for the exact Lagrangian $V \subset T^*(M \times \bR)$ with primitive $\lambda_{T^*(M \times \bR_s)}|_V = df_V$, the conification is defined by
\begin{equation*}
    \Lambda_V = \lc (x, s, t; \xi, \sigma, \tau) \mid (x, s; \xi/\tau, \sigma/\tau) \in V, \, t = - f_V(x, s; \xi/\tau, \sigma/\tau), \, \tau > 0 \rc.
\end{equation*}

\begin{theorem}\label{thm:sheaf-quan-cob-filling}
    Let $V \subset T^*(M \times \bR_s)$ be an exact Lagrangian cobordism with conical ends between $(L_0,\dots,L_{m-1})$ and $(L'_0,\dots,L'_{n-1})$  whose Maslov class and relative second Stiefel--Whitney class vanish. 
    Then there is a simple sheaf quantization 
    \begin{equation*}
        F_V \in \Sh_{\Lambda_V}(M \times \bR_s \times \bR_t).
    \end{equation*}
\end{theorem}

Firstly, using the general result, we show the existence of a doubled sheaf quantization.

\begin{proposition}\label{prop:quan-cob-doubling-filling}
    Let $V \subset T^*(M \times \bR_s)$ be an exact Lagrangian cobordism with conical ends between $(L_0,\dots,L_{m-1})$ and $(L'_0,\dots,L'_{n-1})$ whose Maslov class and relative second Stiefel--Whitney class vanish. 
    Then there is a doubled sheaf
    \begin{equation*}
        F_{V,\mathrm{dbl}} \in \Sh_{\Lambda_{V,q} \cup \Lambda_{V,r}}(M \times \bR_s \times \bR_t \times (0, +\infty)),
    \end{equation*}
    which is simple along $\Lambda_{V,q} \cap T^*_{\tau>0}(M \times \bR_s \times \bR_t \times (0, +\infty))$.
\end{proposition}

By \cref{thm:doubling-largeshift}, it suffices to show the following lemma, which immediately follows from \cref{lem:nbhd-filling}.

\begin{lemma}\label{lem:cob-nbhd-filling}
    Let $V \subset T^*(M \times \bR_s)$ be an exact Lagrangian cobordism with conical ends between $(L_0,\dots,L_{m-1})$ and $(L'_0,\dots,L'_{n-1})$. 
    Then there is a tubular neighborhood of $V$ of positive radius with respect to the complete adapted metric.
\end{lemma}
\begin{proof}
    We show that all the three pieces of the Lagrangian $V \cap T^*(M \times (-\infty, -1])$, $V \cap T^*(M \times [+1, +\infty))$ and $V \cap T^*(M \times [-1, 1])$ have tubular neighborhoods of positive radius with respect to the complete adapted metric. 
    
    Consider $V \cap T^*(M \times (-\infty, -1])$ and $V \cap T^*(M \times [+1, +\infty))$. By \cref{lem:nbhd-filling}, we may assume that $L_0, \dots, L_{m-1}$ and $L'_0, \dots, L'_{n-1}$ have tubular neighborhoods $U_{L_0}, \dots, U_{L_{m-1}}$ and $U_{L'_0}, \dots, U_{L'_{n-1}}$. Then the neighborhoods
    \begin{equation*}
        \bigcup_{i=0}^{m-1} U_{L_i} \times (-\infty, -1] \times (i-\varepsilon, i+\varepsilon), \; \bigcup_{j=0}^{n-1} U_{L'_j} \times [+1, \infty) \times (j-\varepsilon, j+\varepsilon)
    \end{equation*}
    have a positive radius with respect to the complete adapted metric. 
    Then consider \linebreak $V \cap T^*(M \times [-1, 1])$. By \cref{lem:nbhd-filling}, we know that it has a tubular neighborhood of positive radius. This completes the proof.
\end{proof}

\cref{prop:quan-cob-doubling-filling} now directly follows from \cref{thm:doubling-largeshift} and \cref{lem:cob-nbhd-filling}.

Following \cite{Gu12,Gu23}, we try to separate the images of the projections of $\Lambda_V \cup T_u(\Lambda_V)$ on $M \times \bR_s \times \bR_t$. For Lagrangians with conical ends, we will need to cut off the projection of the conical ends, and thus we follow \cite{JT_brane}.

\begin{proposition}\label{prop:quan-cob-single-bounded-filling}
    Let $V \subset T^*(M \times \bR_s)$ be an exact Lagrangian cobordism with conical ends between $(L_0,\dots,L_{m-1})$ and $(L'_0,\dots,L'_{n-1})$ whose Maslov class and relative second Stiefel--Whitney class vanish. Then there is a simple sheaf
    \begin{equation*}
        F_{V \cap T^*(M \times [-R, R])} \in \Sh_{\Lambda_V}(M \times [-R, R] \times \bR_t).
    \end{equation*}
    In addition, for the inclusion $i_{R,R'} \colon M \times [-R, R] \times \bR_t \hookrightarrow M \times [-R', R'] \times \bR_t$,
    \begin{equation*}
        i_{R,R'}^{-1} F_{V \cap T^*(M \times [-R', R'])} \simeq F_{V \cap T^*(M \times [-R, R])}.
    \end{equation*}
\end{proposition}
\begin{proof}
    By \cref{prop:quan-cob-doubling-filling}, there exists a doubled sheaf
    \begin{equation*}
        F_{V, \mathrm{dbl}} \in \Sh_{\Lambda_{V,q} \cup \Lambda_{V,r}}(M \times \bR_s \times \bR_t \times (0, +\infty)_u),
    \end{equation*}
    or equivalently by \cref{thm:GKS} due to \cite{GKS} a sheaf for each time slice $u \in (0, +\infty)$
    \begin{equation*}
        F_{V,\mathrm{dbl},u} \in \Sh_{\Lambda_V \cup T_u(\Lambda_V)}(M \times \bR_s \times \bR_t).
    \end{equation*}
    Since $\Lambda_V \cap \rT(M \times [-R, R] \times \bR_t) /\bR_{>0}$ is upper exact with asymptotically conical ends, there exists a sufficiently large $c > 0$ and $u > c$ such that
    \begin{equation*}
        \pi_{M \times \bR \times [-R,R]}(T_{u} (\Lambda_V) \cap T^*(M \times [-R, R] \times \bR_t)) \subset M \times [-R, R] \times (c + 1, +\infty).
    \end{equation*}
    Moreover, \cite[Subsection~3.6, Lem.~(a)]{JT_brane} implies that for sufficiently large $c > 0$, the Lagrangian $\Lambda_V \cap T^*(M \times [-R, R] \times (c-1, +\infty))$ is the Lagrangian movie of the conical end
    \begin{equation*}
        V_\infty \times \bR_{>0} \cap T^*(M \times [-R, R])
    \end{equation*}
    with respect to the time parameter $(c-1, +\infty)$. Write $j_c \colon M \times [-R, R] \times (-\infty, c) \hookrightarrow M \times [-R, R] \times \bR_t$ and $i_R \colon M \times [-R, R] \times \bR_t \hookrightarrow M \times \bR_s \times \bR_t$ for the inclusion maps. Then we know that
    \begin{equation*}
        \mathring{\MS}(j_c^{-1}i_R^{-1} F_{V,\mathrm{dbl},u}) \subset \Lambda_V \cap T^*(M \times [-R, R] \times \bR_t).
    \end{equation*}
    Consider a diffeomorphism $\rho_c \colon M \times [-R, R] \times (-\infty, c) \rightarrow M \times [-R, R] \times \bR_t$ that is identity on $M \times [-R, R] \times (-\infty, c-1)$. 
    Then consider a homogeneous Hamiltonian isotopy $\phi_z^c \colon \rT(M \times [-R, R] \times \bR_t) \to \rT(M \times [-R, R] \times \bR_t)$ for time $z \in [0, 1]$ that is identity on $M \times [-R, R] \times (-\infty, c-1)$ such that
    \begin{equation*}
        \phi_1^c \circ d\rho_c^{-1} (\Lambda_V \cap T^*(M \times [-R, R] \times (-\infty, c)) = \Lambda_V,
    \end{equation*}
    where $d\rho_c^{-1} \colon T^*(M \times [-R, R] \times (-\infty, c)) \to T^*(M \times [-R, R] \times \bR_t)$ is the pull-back map of the smooth map $\rho_c^{-1}$ on the cotangent bundle. Then $F_{V \cap T^*(M \times [-R, R])} = K(\phi_1^{c}) \circ {\rho_c}_* j_c^{-1}i_R^{-1} F_{V,\mathrm{dbl},u}$ defines a simple sheaf quantization. 
    
    Finally, we check the compatibility of the sheaves with respect to the embeddings $i_{R,R'} \colon M \times [-R, R] \times \bR_t \hookrightarrow M \times [-R', R'] \times \bR_t$, namely 
    \begin{equation*}
        K(\phi_1^{c}) \circ {\rho_c}_* j_c^{-1}i_R^{-1} F_{V,\mathrm{dbl},u} \simeq i_{R,R'}^{-1} \big(K(\phi^{c'}) \circ  {\rho_{c'}}_* j_{c'}^{-1}i_{R'}^{-1} F_{V,\mathrm{dbl},u'} \big),
    \end{equation*}
    where we may assume that $u \leq u'$ and $c \leq c'$. Let $j_{c,c'} \colon M \times [-R', R'] \times (-\infty, c) \hookrightarrow M \times [-R', R'] \times (-\infty, c')$ be the embedding. 
    Write
    \begin{align*}
        \hat i_{R} \colon M \times [-R, R] \times \bR_t \times [u, u'] & \hookrightarrow M \times \bR_s \times \bR_t \times [u, u'], \\
        \hat j_{c} \colon M \times [-R, R] \times (-\infty, c) \times [u, u'] & \hookrightarrow M \times [-R, R] \times \bR_t \times [u, u'].
    \end{align*}
    Then we know that 
    \begin{equation*}
        \hat j_{c}^{-1} \hat i_{R}^{-1} F_{V, \mathrm{dbl}} \in \Sh_{\Lambda_{V,q}}(M \times [-R, R] \times (-\infty, c) \times [u, u']).
    \end{equation*} 
    By \cite[Prop.~5.4.5]{KS90} the cohomology of $\hat j_{c}^{-1} \hat i_{R}^{-1} F_{V, \mathrm{dbl}}$ is locally constant along $[u, u']$. Therefore, by restricting to slices,
    \begin{equation*}
        j_c^{-1}i_R^{-1} F_{V,\mathrm{dbl},u} \simeq j_{c,c'}^{-1}i_{R,R'}^{-1} \big( j_{c'}^{-1}i_{R'}^{-1} F_{V, \mathrm{dbl}, u'} \big).
    \end{equation*}
    Then the compatibility follows from the fact that the restrictions
    \begin{align*}
        \Sh_{\Lambda_q}(M \times [-R, R] \times \bR_t) & \to \Sh_{\Lambda_q}(M \times [-R, R] \times (-\infty, c')), \\
        \Sh_{\Lambda_q}(M \times [-R, R] \times \bR_t) & \to \Sh_{\Lambda_q}(M \times [-R, R] \times (-\infty, c))
    \end{align*}
    are equivalences with inverses given by the GKS composition with $K(\phi_1^c)$ and $K(\phi_1^{c'})$ together with the push forward ${\rho_c}_*$ and ${\rho_{c'}}_*$, whose restriction to $M \times [-R, R] \times (-\infty, c-1)$ is the identity.
\end{proof}

\begin{proof}[Proof of \cref{thm:sheaf-quan-cob-filling}]
    For any $R > 0$, by \cref{prop:quan-cob-single-bounded-filling}, there exists a simple sheaf
    \begin{equation*}
        F_{V \cap T^*(M \times [-R, R])} \in \Sh_{\Lambda_V}(M \times [-R, R] \times \bR_t)
    \end{equation*}
    compatible with respect to $i_{R,R'} \colon M \times [-R, R] \times \bR_t \hookrightarrow M \times [-R', R'] \times \bR_t$. Hence we can construct a simple sheaf $F_V \in \Sh_{\Lambda_V}(M \times \bR_s \times \bR_t)$ such that $i_R^{-1}F_V = F_{V \cap T^*(M \times [-R, R])}$.
\end{proof}

\begin{remark}\label{rem:wrap-quantization}
    Alternatively, we can apply the positive wrapping functor of Kuo \cite[Thm.~1.2]{Kuo23} given by taking colimits for all positive Hamiltonians supported away from $\Lambda$. This will help us avoid taking truncations at a finite portion $M \times [-R, R] \times \bR_t$. More precisely, we can consider a Hamiltonian flow $\phi_z^c \colon \rT(M \times \bR_s \times \bR_t) \to \rT(M \times \bR_s \times \bR_t)$ such that $\phi_z^c(\Lambda_V) = \Lambda_V$ and $\phi_z^c(T_u(\Lambda_V)) = T_{u+z}(\Lambda_V)$. Then $\phi_z^c$ defines a cofinal positive Hamiltonian flow in the complement of $\Lambda_V$ in the sense of \cite[Lem.~3.31]{Kuo23}, and by \cite[Thm.~1.2]{Kuo23}, we know
    \begin{equation*}
        F_V \coloneqq \operatorname{colim}_{z \in \bR} K(\phi_z^c) \circ F_{V,\text{dbl},u} \in \Sh_{\Lambda_V}(M \times \bR_s \times \bR_t).
    \end{equation*}
    Furthermore, one can show that $F_V$ is simple along $\Lambda_V$ as $F_{V,\text{dbl},u}$ is simple along $\Lambda_V$. This provides another proof of \cref{thm:sheaf-quan-cob-filling}. However, we will need the existence of the integration flow for non-compactly supported Hamiltonians, which requires the choice of a complete adapted metric on the non-compact manifold $T^*(M \times \bR)$. 
\end{remark}

\begin{remark}\label{rem:quan-jintreumann}
    Jin--Treumann \cite[Subsection 1.11.4]{JT_brane} have stated that such a sheaf quantization theorem should hold after deforming the Lagrangian cobordism so that it becomes lower exact and satisfies \cite[Assumption~3.7]{JT_brane}. 
    However, $V$ is a product Lagrangian outside compact subset that is not eventually conical with respect to the standard Liouville form on $T^*(M \times \bR_s)$. 
    One will need to carefully choose a Hamiltonian isotopy on $T^*(M \times \bR_s)$ to verify the assumption (due to a similar issue mentioned after \cite[Assumption~3.7]{JT_brane}, the assumption does not hold for arbitrary deformation that makes the cobordism lower exact). 
    Our proof here relies on \cref{thm:doubling-largeshift}, which is independent of \cite[Assumption~3.7]{JT_brane}, though we assume lower/upper exactness to ensure that $\Lambda_V$ is properly embedded in $\rT(M \times \bR_s \times \bR_t)$.
\end{remark}

\cref{thm:sheaf-quan-cob-filling} allows us to define the sheaf quantization of exact Lagrangian cobordisms in the Tamarkin category.

\begin{corollary}\label{cor:sheaf-quan-cob-filling}
    Let $V \subset T^*(M \times \bR_s)$ be an exact Lagrangian cobordism with conical ends between $(L_0,\dots,L_{m-1})$ and $(L'_0,\dots,L'_{n-1})$ whose Maslov class and relative second Stiefel--Whitney class vanish. Then there is a simple sheaf quantization 
    \begin{equation*}
        F_V \in \cT_V(T^*(M\times \bR_s)).
    \end{equation*}
\end{corollary}

\section{Lagrangian shadows and interleaving distance}\label{section:shadow}

In this section, we prove an iterated cone decomposition of sheaves and the interleaving distance bound of sheaves by the shadow. 

For notational convenience, we set $J=\bR_s$. 
Throughout this section, we let $F \in \cT(T^*(M \times J))$, where $\RS(F) = \overline{\rho(\MS(F) \cap \Omega_+)}$ is not necessarily Lagrangian, and define 
\begin{equation*}
    \Pi(F) \coloneqq \overline{ p(\RS(F)) } = \overline{ p(\rho( \MS(F) \cap \Omega_+)) } \subset T^*J,
\end{equation*}
where $p \colon T^*(M \times J) \to T^*J$ is the projection.
We assume the following condition until the end of this section, which essentially means that $\RS(F)$ is a singular coisotropic cobordism.

\begin{assumption}\label{assumption:proj}
    The (coisotropic) set $\Pi(F)$ satisfies 
    \begin{align*}
        \Pi(F) \cap T^*(-\infty,-1] & \subset (-\infty,-1] \times \bigcup_{i=0}^{m-1} \{i\}, \\
        \Pi(F) \cap T^*[+1,+\infty) & \subset [+1,+\infty) \times \bigcup_{j=0}^{n-1} \{j\}.
    \end{align*}
\end{assumption}

\subsection{Twisted complex and cone decomposition}

For $S \subset J$, we denote  $F|_{M \times S \times \bR_t}$ by $F_S$. 
We shall show that $F_{-1}=F_{\{-1\}}$ and $F_{+1}=F_{\{+1\}}$ admit iterated cone decompositions obtained through the microlocal cut-off functor. 

In general, for $H \in \cT(T^*X)$, we say that $H$ admits an iterated cone decomposition of $F^0, F^1, \dots, F^{k-1}, F^k$ and write 
\begin{equation*}
    H \simeq \ld F^0 \to F^1 \to \dots \to F^{k-1} \to F^{k} \rd
\end{equation*}
if there also exists $H \simeq E^{-1}, E^0, E^1 , \dots, E^{k-1}, E^k \simeq 0$ such that we have exact triangles
\begin{equation*}
    F^i \to E^{i-1} \to E^{i} \toone
\end{equation*}
for $i=0,\dots,k$. In particular, $H \simeq \ld F^0 \to F^1 \to \dots \to F^k \rd$ means that 
\begin{equation*}
    H \simeq \Cone(\Cone( \cdots \Cone(F^0[-k] \to F^1[1-k]) \to \dots  \to F^{k-1}[-1]) \to F^k).
\end{equation*}
This is equivalent to a twisted complex in $\cT(T^*X)$ introduced in \cite{BondalKapranov} for general dg-categories. 
We will use this standard fact in Sect.~\ref{subsec:lagrangian-intersection}.

We can show that the above condition of iterated cone decomposition for $H$ is equivalent to the following one: there exists $H \simeq G^{-1}, G^0, G^1, \dots, G^{k-1}, G^k \simeq 0$ such that we have exact triangles
\begin{equation*}
    G^i \to G^{i-1} \to F^i \toone
\end{equation*}
for $i=0,\dots,k$. In other words,
\begin{equation*}
    H \simeq \Cone(F^0[-1] \to \Cone(F^1[-1] \to \dots \to \Cone(F^{k-1}[-1] \to F^k)\cdots)).
\end{equation*}

\begin{lemma}\label{lem:iterated-cone}
    Let $F^0, F^1, \dots, F^{k-1}, F^k \in \cT(T^*X)$. For any iterated cone decomposition
\begin{equation*}
    H \simeq \Cone(\Cone( \cdots \Cone(F^0[-k] \to F^1[1-k]) \to \dots \to F^{k-1}[-1]) \to F^k),
\end{equation*}
    there exist natural morphisms such that
\begin{equation*}
    H \simeq \Cone(F^0[-1] \to \Cone(F^1[-1] \to \dots \to \Cone(F^{k-1}[-1] \to F^k)\cdots)).
\end{equation*}
\end{lemma}
\begin{proof}
    We can consider the following commutative diagram with solid arrows:
    \[
    \begin{tikzcd}[column sep=15pt]
        F^0[-2] \ar[r] \ar[d] & 0 \ar[d] \ar[r] & F^0[-1] \ar[d,dotted] \ar[r] & F^0[-1] \ar[d] \\
        F^1[-1] \ar[r] \ar[d] & F^2 \ar[r] \ar[d] & \Cone(F^1[-1] \to F^2) \ar[r] \ar[d,dotted] & F^1 \ar[d] \\
        \Cone(F^0[-2] \to F^1[-1]) \ar[r,dotted] \ar[d] & F^2 \ar[r,dotted] \ar[d,dotted] & H \ar[r,dotted] \ar[d] & \Cone(F^0[-1] \to F^1) \ar[d] \\
        F^0[-1] \ar[r] & 0 \ar[r] & F^0 \ar[r] & F^0.
     \end{tikzcd}
    \]
    Then by \cite[Exercise~10.6]{KS06} in the triangulated setting and \cite[Lem.~3.6 and Example~(3)]{AnnoLogvinenko} in the dg enhanced setting, the dotted arrows can be completed so that the right bottom square is anti-commutative, all the other squares are commutative, and all the rows and columns are exact triangles 
    In particular, we have
    \begin{align*}
        H &\simeq \Cone(F^0[-1] \to \Cone(F^1[-1] \to F^2)) \\
        &\simeq \Cone(\Cone(F^0[-2] \to F^1[-1]) \to F^2).
    \end{align*}
    Thus we can complete the proof by induction.
\end{proof}

The main result in this section is an iterated cone decomposition at the two ends for sheaves $F \in \cT(T^*(M \times J))$ that satisfies \cref{assumption:proj}, where we recall that $J = \bR_s$.

\begin{theorem}\label{theorem:iterated_cone}
    Suppose $F \in \cT(T^*(M \times J))$ satisfies \cref{assumption:proj}. There exist sheaves $F_{-1}^i \in \cT(T^*M)$ for $i = 0, \dots, m-1$ such that
    \begin{equation*}
        F_{(-\infty,-1]}
        \simeq 
        \ld \pi_0^{-1}F_{-1}^0 \to \pi_1^{-1}F_{-1}^1 \to \dots \to \pi_{m-2}^{-1}F_{-1}^{m-2} \to \pi_{m-1}^{-1}F_{-1}^{m-1} \rd,
    \end{equation*}
    and sheaves $F_{+1}^j \in \cT(T^*M)$ for $j = 0, \dots, n-1$ such that
    \begin{equation*}
        F_{[+1, +\infty)}
        \simeq 
        \ld \pi_{n-1}^{-1}F_{+1}^{n-1} \to \pi_{n-2}^{-1}F_{+1}^{n-2} \to \dots \to \pi_1^{-1}F_{+1}^1 \to \pi_0^{-1}F_{+1}^0 \rd, 
    \end{equation*}
    where $\pi_i \colon M \times J \times \bR_t \to M \times \bR_t; (x,s,t) \mapsto (x,t-is)$.
    In particular, 
    \begin{align*}
        F_{-1} \simeq  
        \ld F_{-1}^0 \to T_{1}F_{-1}^1 \to \dots \to T_{m-2}F_{-1}^{m-2} \to T_{m-1}F_{-1}^{m-1} \rd, \\
        F_{+1} \simeq \ld T_{1-n}F_{+1}^{n-1} \to T_{2-n}F_{+1}^{n-2} \to \dots \to T_{-1}F_{+1}^1 \to F_{+1}^0 \rd.
    \end{align*}
\end{theorem}

\begin{figure}[tbh]
    \centering
    \includegraphics[width=0.5\textwidth]{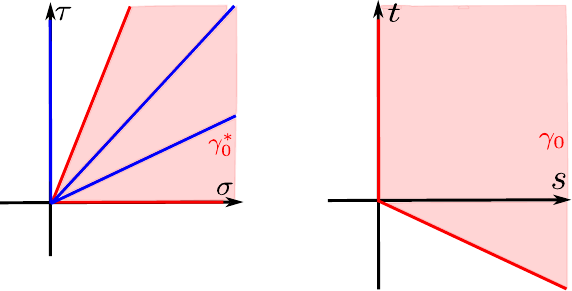}
    \caption{The cone $\gamma_i$ and dual cone $\gamma_i^*$ for $i = 0, 1, \dots, m$ in the proof of \cref{theorem:iterated_cone}.}
    \label{fig:cutoff}
\end{figure}

\begin{proof}
    Consider the assertion for $F_{(-\infty, -1]}$ and $F_{-1}$. For $i=0, 1,\dots, m$, we define a sequence of cones $\gamma_0 \supset \gamma_1 \supset \dots \supset \gamma_m$ in $\bR^2$ and their dual cones $\gamma_0^* \subset \gamma_1^* \subset \dots \subset \gamma_{m}^*$ in the dual space $\bR^2$ by
    \begin{equation*}
        \gamma_i \coloneqq \lc (s, t) \relmid 0 \le s, -\lb i + \frac{1}{2} \rb s \le t \rc, \; \gamma_i^* \coloneqq \lc (\sigma, \tau) \relmid \sigma \ge \lb i + \frac{1}{2} \rb \tau \ge 0 \rc.
    \end{equation*}
    Set $G^{-1} \coloneqq F$. Using \cref{eq:star} we iteratively define 
    \begin{equation*}
        G^i \coloneqq \bfk_{\gamma_i} \star G^{i-1} 
        \quad \text{and} \quad 
        F^i \coloneqq \bfk_{\gamma_i \setminus \{0\}} \star G^{i-1}.
    \end{equation*}
    Since the cone $\gamma_i$ is supported in $\{(s, t) \mid s \ge 0\}$ as in \cref{fig:cutoff}, we find that the stalks of $\bfk_{\gamma_i} \star G^{i-1}$ at $\{(s, t) \mid s = s_0\}$ depends only on the sections of $G^{i-1}$ on $\{(s, t) \mid s \leq s_0\}$.
    This implies that $G^i_{(-\infty,-1]}$ depends only on $G^{i-1}_{(-\infty,-1]}$.
    Consider the exact triangle
    \begin{equation*}
        G^i_{(-\infty,-1]} \to G^{i-1}_{(-\infty,-1]} \to F^i_{(-\infty, -1]} \toone.
    \end{equation*}
    By the microlocal cut-off lemma, we know that
    \begin{equation*}
        \mathring{\MS}(G^i_{(-\infty, -1]}) 
        \subset \mathring{\MS}(G^{i-1}_{(-\infty, -1]}) \cap T^*M \times ((-\infty,-1] \times \bR_t) \times \gamma_i^*.
    \end{equation*}
    Since $\gamma_0^* \subset \gamma_1^* \subset \dots \subset \gamma_{m}^*$, we can inductively show that
    \begin{equation*}
        \mathring{\MS}(G^i_{(-\infty, -1]}) \subset T^*M \times ((-\infty,-1] \times \bR_t) \times \bigcup_{\alpha=i}^{m-1} \{ (\sigma,\tau) \mid \sigma = \alpha \tau \}.
    \end{equation*}
    In particular, when $i = m$, $\mathring{\MS}(G^m_{(-\infty, -1]}) = \varnothing$,
    we know that $G^m_{(-\infty, -1]} \simeq 0$. Hence by \cref{lem:iterated-cone} we get
    \begin{equation*}
        F_{(-\infty, -1]} 
        \simeq 
        \ld F^0_{(-\infty, -1]} \to F^1_{(-\infty, -1]} \dots \to F^{m-2}_{(-\infty, -1]} \to F^{m-1}_{(-\infty, -1]} \rd.
    \end{equation*}
    
    Finally, we consider $F^i_{(-\infty, -1]}$ and conclude that it is of the form 
    \begin{equation*}
        F^i_{(-\infty, -1]} \simeq \pi^{-1}_iF^i_{-1}
    \end{equation*}
    for some $F^i_{-1} \in \cT(T^*M)$. 
    By microlocal cut-off lemma, $G^i_{(-\infty, -1]} \to G^{i-1}_{(-\infty, -1]}$ is an isomorphism on $T^*M \times ((-\infty,-1] \times \bR_t) \times \Int(\gamma_i)$. 
    Thus we can show that
    \begin{equation*}
        \mathring{\MS}(F^i_{(-\infty, -1]}) \subset \mathring{\MS}(G^{i-1}_{(-\infty, -1]}) \setminus \mathring{\MS}(G^{i}_{(-\infty, -1]}) \subset \{ (\sigma,\tau) \mid \sigma = i\sigma \}.
    \end{equation*}
    By \cite[Prop.~5.4.5]{KS90} the cohomology of $F^i_{(-\infty, -1]}$ is locally constant along the fiber of the projection $\pi_i \colon M \times (-\infty, -1] \times \bR \to M \times \bR, \, (x, s, t) \mapsto (x, s-it)$, which proves that $F^i_{(-\infty, -1]} \simeq \pi^{-1}_i \pi_{i*}F^i_{(-\infty, -1]}$. This implies the assertion.
    
    Then consider the assertion for $F_{[+1, +\infty)}$ and $F_{+1}$. For $j=0, 1,\dots, n$, we define a sequence of cones $\gamma'_0 \supset \gamma'_1 \supset \dots \supset \gamma'_n$ in $\bR^2$ and their dual cones $\gamma_0^{\prime*} \subset \gamma_1^{\prime*} \subset \dots \subset \gamma_{n}^{\prime*}$ in the dual space $\bR^2$ by
    \begin{equation*}
        \gamma'_j \coloneqq \lc (s, t) \relmid t \ge -\lb n - i - \frac{1}{2} \rb s \ge 0 \rc, \; \gamma_j^{\prime*} \coloneqq \lc (\sigma, \tau) \relmid \tau \ge 0, \lb n - i - \frac{1}{2} \rb \tau \ge \sigma \rc.
    \end{equation*}
    Then the same microlocal cut-off argument will imply the assertion.
\end{proof}

\begin{remark}
    We can also explain the microlocal cut-off argument using sheaf theoretic wrapping developed by Kuo~\cite[Thm.~1.2]{Kuo23} with suitable generalizations; see \cref{rem:wrap-quantization}. We again need the existence of the integration flow for non-compactly supported Hamiltonians, which requires the choice of a complete adapted metric on the non-compact manifold $T^*(M \times (-\infty, -1))$ or $T^*(M \times (+1, +\infty))$.
\end{remark}

\subsection{Shadow distance and interleaving distance}
In this subsection, we study the relation between the shadow distance of cobordisms and the interleaving distance of their sheaf quantizations.

Recall that $J=\bR_s$. 
Consider $F \in \cT(T^*(M \times J))$ that satisfies \cref{assumption:proj}. For $\Pi(F) = \overline{ p(\RS(F)) } = \overline{p(\rho (\MS(F) \cap \Omega_+))} \subset T^*J$, let $\cU(F)$ be the union of unbounded regions of $T^*J \setminus \Pi(F)$, and
\begin{equation*}
    A(F) \coloneqq T^*J \setminus \cU(F).
\end{equation*}
The shadow of $F$ is defined as 
\begin{equation*}
	\cS(F) 
	\coloneqq 
	\Area(A(F)) = \Area(T^*J \setminus \cU(F)),
\end{equation*}
where $\Area(A)$ is the Lebesgue measure of $A \subset \bR^2$. In other words, we define the shadow of a sheaf $F$ that satisfies \cref{assumption:proj} to be the shadow of the subset $\RS(F)$ as introduced in \cite{CS19}.
When $F_V \in \cT_V(T^*(M \times J))$ is the sheaf quantization of a Lagrangian cobordism $V$, the shadow $\cS(F_V)$ is by definition the shadow of the cobordism $V$ in \cite{CS19,BCS18}.

We need the following technical lemma to prove the result in this section.

\begin{lemma}[cf.~{\cite[Section~3.11]{CS19}}]\label{lem:deform-shadow}
    Let $A$ be a closed subset of $T^*J$ such that $A \cap T^*[-1, 1]$ is bounded, all connected components of $T^*J \setminus A$ have infinite area, and
    \begin{equation*}
        A \cap T^*(-\infty,-1] \subset (-\infty,-1] \times \bigcup_{i=0}^{m-1} \{i\}, \quad
        A \cap T^*[1,\infty) \subset [1,\infty) \times \bigcup_{j=0}^{n-1} \{j\}.
    \end{equation*}
    For any $\varepsilon>0$, there exists $T > 0$, a Hamiltonian isotopy $\phi^\varepsilon_z \colon T^*J \rightarrow T^*J$ supported in $J \times [-T,T]$
    and a subset $A_\varepsilon = \{(s, \sigma) \mid -1 \leq s \leq 1, 0 \leq \sigma \leq \sigma_+(s)\}$ bounded by continuous sections on $T^*[-1, 1]$ such that:
    \begin{enumerate}
    \renewcommand{\labelenumi}{$\mathrm{(\arabic{enumi})}$}
        \item $\phi^\varepsilon_1(A) \cap T^*[-1, 1] \subset A_\varepsilon$ and $\mathrm{Area}(A_\varepsilon) \leq \mathrm{Area}(A) + \varepsilon$ and
        \item $\phi^\varepsilon_z(A \setminus T^*[-1, 1]) = A \setminus T^*[-1, 1]$.
    \end{enumerate}
    In particular, the corresponding Hamiltonian function $H_z$ is supported in $J \times [-T, T]$.
\end{lemma}
\begin{proof}
    Since $A \cap T^*[-1, 1]$ is a measurable subset, there exists a countable collection of squares $\{D_\alpha\}_{\alpha \in \bN}$ that covers $A$ and 
    \begin{equation*}
        \Area\lb\bigcup\nolimits_{\alpha \in \bN} D_\alpha\rb \leq \Area(A)+ \frac{\varepsilon}{3}.
    \end{equation*}
    Since $A \cap T^*[-1, 1]$ is a compact subset, we can moreover conclude that there exists a finite subcover $\{D_\alpha\}_{1 \leq \alpha \leq N} \subset \{D_\alpha\}_{\alpha \in \bN}$ such that
    \begin{equation*}
        \Area \lb \bigcup\nolimits_{1 \leq \alpha \leq N} D_\alpha \rb \leq \Area(A)+ \frac{\varepsilon}{3}.
    \end{equation*}
    Note that the union $\bigcup_{1 \leq \alpha \leq N} D_\alpha$ has a piecewise smooth boundary, and we can assume it to be simply connected.
    
    Consider a non-negative section $\sigma_+$ such that $\sigma_+(-1) \geq m, \sigma_+(+1) \geq n$ and the domain $A_\varepsilon = \{(s, \sigma) \mid -1 \leq s \leq 1, 0 \leq \sigma \leq \sigma_+(s) \}$ satisfies the area inequality
    \begin{equation*}
         \Area \lb \bigcup\nolimits_{1 \leq \alpha \leq N} D_\alpha \rb \leq \Area(A_{\varepsilon}) \leq \Area \lb \bigcup\nolimits_{1 \leq \alpha \leq N} D_\alpha \rb + \frac{\varepsilon}{3}.
    \end{equation*}
    Since all components $T^*J \setminus A$ have infinite area, we may assume that all components of $T^*J \setminus \lb A \cup \bigcup_{1 \leq \alpha \leq N} D_\alpha \rb$ have infinite area as well. Then we can find a piecewise smooth area preserving isotopy $\phi_z^\varepsilon \colon T^*J \to T^*J$ fixed on $A \setminus T^*[-1, 1]$ such that
    \begin{equation*}
        \phi_1^\varepsilon \lb \bigcup\nolimits_{1 \leq \alpha \leq N} D_\alpha \rb \subset A_{\varepsilon}.
    \end{equation*}
    Since piecewise smooth diffeomorphisms can be approximated by smooth diffeomorphisms, we thus finish the construction. 
    See \cref{fig:estimate} for a schematic picture of the proof.
\end{proof}

\begin{figure}[ht]
    \centering
    \includegraphics[width=0.95\textwidth]{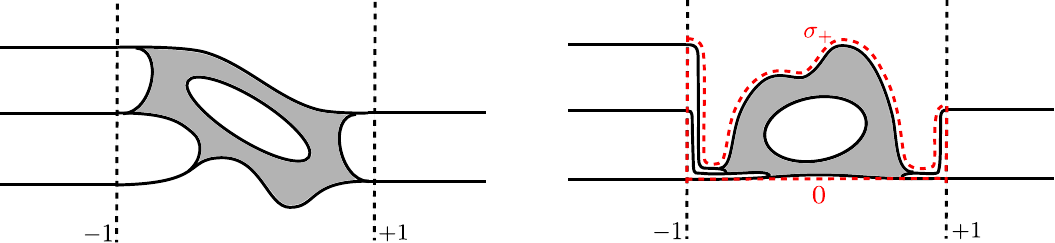}
    \caption{The schematic picture of the Hamiltonian isotopy in \cref{lem:deform-shadow} such that the Lagrangian shadow is bounded in between the sections $0$ and $\sigma_+ \colon [-1, 1] \to T^*[-1, 1]$.}
    \label{fig:estimate}
\end{figure}

Now we can prove the main results on the distance bound. First, we prove the distance bound when each positive and negative part has one end.

\begin{proposition}[One end for each part]\label{proposition:inequality_one_end}
    Suppose $F \in \cT(T^*(M \times J))$ satisfies
    \begin{equation*}
        \Pi(F) \cap T^*((-\infty, -1] \cup [1, +\infty)) = (-\infty, -1] \cup [1, +\infty) \times \{0\}.
    \end{equation*}
    Then there exists $c \in \bR$ such that 
    \begin{equation*}
        d'_{\cT(T^*M)}(F_{-1},T_c F_{+1}) \le \cS(F).
    \end{equation*}
    In particular, $F_{-1} \simeq F_{+1}$ in $\cT_\infty(T^*M)$.
\end{proposition}
\begin{proof}
    Set $A \coloneqq A(F)$.
    Consider the Hamiltonian isotopy $\phi^\varepsilon_z \colon T^*J \to T^*J$, $z \in [0, 1]$, supported in $J \times [-T,T]$ in \cref{lem:deform-shadow}. 
    Consider the GKS kernel $K^\varepsilon$ of the Hamiltonian isotopy $\id_{T^*M} \times \phi^\varepsilon_z$ on $T^*(M \times J)$ obtained by \cref{rem:noncompact-gks} and write $K^\varepsilon_z = K^\varepsilon|_{M \times \bR \times M \times \bR \times \{z\}}$. 
    We have
	\begin{equation*}
		\mathring{\MS}(K^\varepsilon_1 \circ F) \subset 
		T^*M \times 
		\{ (s,t,\sigma,\tau) \mid (s;\sigma/\tau) \in \phi_1^\varepsilon(A) \}.
	\end{equation*}
    By \cref{lem:deform-shadow}, there exist a non-negative section $\sigma_+ \colon J \to T^*J$ such that
	\begin{equation*}
        \{ (t,s;\tau,\sigma) \mid (s;\sigma/\tau) \in \phi_1^\varepsilon(A) \}
        \subset
        \{(t,s;\tau,\sigma) \mid 0 \le \sigma \le \sigma_+(s) \cdot \tau \},
    \end{equation*}
	and for $A_\varepsilon = \{(s, \sigma ) \mid 0 \leq \sigma \leq \sigma_+(s) \}$, we have $\Area(A_\varepsilon) \leq \Area(A) + \varepsilon$. Thus, applying \cref{prp:htpy_wabisom}, we obtain that
    \begin{equation}\label{eq:distance-oneend}
        d'_{\cT(T^*M)}(i_{-1}^{-1}(K^\varepsilon_1 \circ F), i_{1}^{-1}(K^\varepsilon_1 \circ F)) \leq \Area(A_\varepsilon) \leq \Area(A) + \varepsilon.
    \end{equation}
	As $\phi^\varepsilon_z(A \setminus T^*[-1, 1]) = A \setminus T^*[-1, 1]$, we can find that the reduced microsupports satisfy
     \begin{equation*}
	    \RS(K^\varepsilon_z \circ F) \setminus T^*(M \times [-1, 1]) = \RS(F) \setminus T^*(M \times [-1, 1]).
	\end{equation*}
    As the projection by $\rho$ determines the conical Lagrangian up to height in $\bR_t$, this implies that $i_{\pm 1}^{-1}(K^\varepsilon_z \circ F)$ are isomorphic to $F_{\pm 1}$ up to some vertical translation. Let $c_\varepsilon$ be the area with signs enclosed by the curve $0_{[-1,1]}$ and ${(\phi^\varepsilon_1)^{-1}}(0_{[-1,1]})$. Then by applying Stokes' formula to \eqref{eqn:hamilton-lift},
	\begin{equation*}
	    i_{-1}^{-1}(K^\varepsilon_1 \circ F) \simeq F_{-1}, \quad i_{+1}^{-1}(K^\varepsilon_1 \circ F) \simeq T_{c_\varepsilon} F_{+1}.
	\end{equation*}
    Finally, since \cref{eq:distance-oneend} holds for any $\varepsilon \in \bR_{>0}$ and $c_\varepsilon$ converges to the signed area enclosed by the lower boundary of $A$ and the zero section as shown in \cref{fig:single}, we obtain the result.
\begin{figure}[ht]
    \centering
    \includegraphics[width=0.85\textwidth]{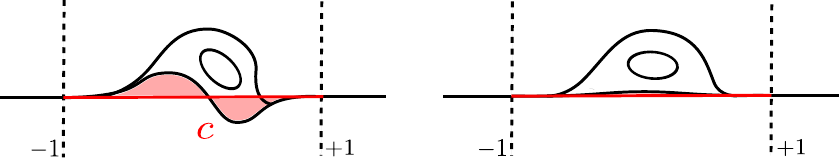}
    \caption{The shift $c$ as the signed area of the red region in \cref{proposition:inequality_one_end}.}
    \label{fig:single}
\end{figure}
\end{proof}

\begin{remark}
    From the proof, one can see that the choice $c \in \bR$ is not unique. In fact, letting $c_-$ be the area with signs enclosed by the lower boundary of $A$ and the zero section and $c_+$ be the area with signs enclosed by the upper boundary of $A$ and the zero section, one can prove the result for any $c \in [c_-, c_+]$ where $c_+ - c_- = \cS(F)$.
\end{remark}

Now we state the distance bound in the general case.

\begin{theorem}[General case]\label{theorem:inequality_multiple_ends}
    Suppose $F \in \cT(T^*(M \times J))$ satisfies \cref{assumption:proj}, and the iterated cone decompositions for $F_{-1}$ and $F_{+1}$ are given as in \cref{theorem:iterated_cone}.
    There exist increasing sequences of real numbers $\{c_i\}_{i=0,\dots,m-1}, \{c'_j\}_{j=0,\dots,n-1}$ such that 
    \begin{equation*}
        d'_{\cT(T^*M)}(\tilde{F}_{-1},\tilde{F}_{+1}) \le \cS(F),
    \end{equation*}
    where 
    \begin{align*}
        \tilde{F}_{-1} & \coloneqq 
        \ld F_{-1}^0 \to T_{c_1}F_{-1}^1 \to \dots \to T_{c_{m-2}}F_{-1}^{m-2} \to T_{c_{m-1}}F_{-1}^{m-1} \rd, \\ 
        \tilde{F}_{+1} & \coloneqq 
        \ld T_{-c'_{n-1}} F_{+1}^{n-1} \to T_{-c'_{n-2}}F_{+1}^{n-2} \to \dots \to T_{-c'_{1}}F_{+1}^{1} \to T_{-c'_{0}}F_{+1}^{0} \rd. 
    \end{align*}
    In particular, one has $\tilde{F}_{-1} \simeq \tilde{F}_{+1}$ in $\cT_\infty(T^*M)$.
\end{theorem}
\begin{proof}
    Set $A \coloneqq A(F)$. Consider the Hamiltonian isotopy $\phi^\varepsilon_z \colon T^*J \to T^*J$ in \cref{lem:deform-shadow} and the GKS kernel $\cK^\varepsilon$ of $\id_{T^*M} \times \phi^\varepsilon_z$ on $T^*(M \times J)$ obtained by \cref{rem:noncompact-gks}. Again, by \cref{lem:deform-shadow} and \cref{prp:htpy_wabisom}, we have
    \begin{equation}\label{eq:distance-multiple-ends}
        d'_{\cT(T^*M)}(i_{-1}^{-1}(K^\varepsilon_1 \circ {F}), i_1^{-1}(K^\varepsilon_1 \circ {F})) \leq \Area(A_\varepsilon) \leq \Area(A) + \varepsilon.
    \end{equation}
    As $\phi^\varepsilon_z(A \setminus T^*[-1, 1]) = A \setminus T^*[-1, 1]$, we can find that the reduced microsupports satisfy
    \begin{equation*}
	    \RS(K^\varepsilon_z \circ F) \setminus T^*(M \times [-1, 1]) = \RS(F) \setminus T^*(M \times [-1, 1]).
    \end{equation*}
    Applying \cref{theorem:iterated_cone} to the sheaf $K^\varepsilon \circ F$, we can write down the iterated cone decompositions 
    \begin{align*}
        i_{-1}^{-1} (K^\varepsilon \circ F) 
        & \simeq \ld \pi_0^{-1}\overline{F}_{-}^0 \to \pi_1^{-1} \overline{F}_{-1}^1 \to \dots \to \pi_{m-2}^{-1}\overline{F}_{-1}^{m-2} \to \pi_{m-1}^{-1}\overline{F}_{-1}^{m-1} \rd, \\ 
        i_{+1}^{-1} (K^\varepsilon \circ F) 
        & \simeq \ld \pi_{n-1}^{-1}\overline{F}_{+1}^{n-1} \to \pi_{n-2}^{-1} \overline{F}_{+1}^{n-2} \to \dots \to \pi_{1}^{-1}\overline{F}_{+1}^{1} \to \pi_{0}^{-1}\overline{F}_{+1}^{0} \rd. 
    \end{align*}
    Write $F_{\pm 1, u}^{j} = \overline{F}_{\pm 1}^j|_{M \times \bR_t \times \{u\}}$. Then the reduced microsupport estimate
    \begin{equation*}
        \RS({F}_{\pm 1,u}^{j}) = \RS(F_{\pm 1}^j) \quad (0 \leq j \leq m-1 \text{ or } n-1)
    \end{equation*}
    implies that $F_{\pm 1, 1}^{j}$ are isomorphic to $F_{\pm 1}^j$ up to vertical shifts ($0 \leq j \leq m-1 \text{ or } n-1$). 
    We determine the vertical shifts for all the components of $F_{-1}^i$, $0 \leq i \leq m-1$, and $F_{+1}^j$, $0 \leq j \leq n-1$ as follows. 
    Write $\ell_{\pm,i} = \{(\pm 1, \sigma) \mid 0 \leq \sigma \leq i\}$. 
    Let $c_{i,\varepsilon}$ be the area of the region enclosed by $\ell_{-,i}$ and $(\phi^\varepsilon)^{-1}(\ell_{-,i})$, and $c'_{j,\varepsilon}$ be the area of the region enclosed by $0_{[-1, 1]} \cup \ell_{+,j}$ and $(\phi^\varepsilon)^{-1}(0_{[-1, 1]} \cup \ell_{+,j})$. 
    Then by applying Stokes' formula to \eqref{eqn:hamilton-lift},
    \begin{equation*}
        F_{-1,1}^i \simeq T_{c_{i,\varepsilon}}F_{-1}^i, \quad F_{+1,1}^j \simeq T_{-c'_{j,\varepsilon}}F_{+1}^j.
    \end{equation*}
    The iterated cone decomposition can be written as 
    \begin{align*}
        i_{-1}^{-1} (K^\varepsilon_1 \circ F) 
        & \simeq \ld F_{-1}^0 \to T_{c_1}F_{-1}^1 \to \dots \to T_{c_{m-2}}F_{-1}^{m-2} \to T_{c_{m-1}}F_{-1}^{m-1} \rd, \\
        i_{+1}^{-1} (K^\varepsilon_1 \circ F) 
        & \simeq \ld T_{-c'_{n-1}} F_{+1}^{n-1} \to T_{-c'_{n-2}}F_{+1}^{n-2} \to \dots \to T_{-c'_{1}}F_{+1}^{1} \to T_{-c'_{0}}F_{+1}^{0} \rd. 
    \end{align*}
    Since \cref{eq:distance-multiple-ends} holds for any $\varepsilon \in \bR_{>0}$, $c_{i,\varepsilon}$ converges to the area enclosed by the boundary of $A$ and $\ell_{-,i}$, and $c_{j,\varepsilon}'$ converges to the signed area enclosed by the lower boundary of $A$ and the zero section plus the area enclosed by the boundary of $A$ and $\ell_{+,j}$ as shown in \cref{fig:multi}, we obtain the result.
\end{proof}

\begin{figure}[htb]
    \centering
    \includegraphics[width=0.95\textwidth]{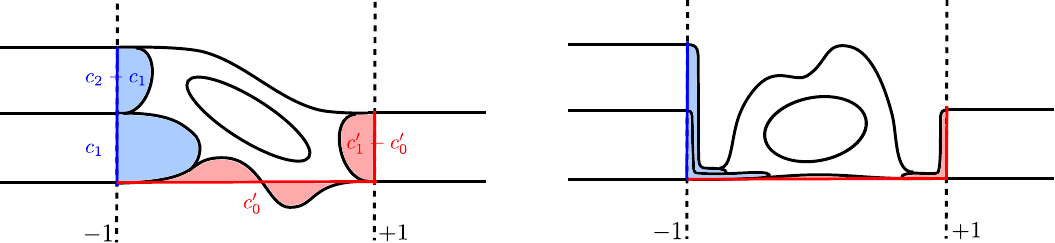}
    \caption{The shifts $c_i$ and $c_j'$ ($i = 1, \dots, m-1$, $j = 0, 1, \cdots, n-1$) in \cref{theorem:inequality_multiple_ends}.}
    \label{fig:multi}
\end{figure}

\subsection{Application to Lagrangian intersection}\label{subsec:lagrangian-intersection}

In this subsection, we apply the distance bound in the previous section to Lagrangian intersection problems. 
Throughout this subsection, we assume that $M$ is compact and $\bfk=\bF_2$,  meaning that all the sheaves are defined over $\bfk = \bF_2$.

The key lemma that relates microlocal theory of sheaves and Lagrangian intersection points is as follows. See for example \cite[Lem.~6.14]{Gu12}, \cite[Lem.~4.9, Prop.~4.10, and Thm.~4.14]{Ike19}, \cite[Lem.~5.19 and Prop.~5.20]{AISQ} and \cite[Cor.~4.7 and Prop.~6.1]{li2021estimating}. 
Recall that $A$ and $B$ intersect cleanly if they are smooth submanifolds in a neighborhood of their intersection and $T_p(A \cap B) = T_pA \cap T_p B$ for any $p \in A \cap B$.
Let $\pi \colon M \times \bR_t \to \bR_t$ be the projection map.
This is the only place in the subsection that we need $\bfk=\bF_2$.

\begin{lemma}[{\cite[Lem.~4.9, Prop.~4.10, and Thm.~4.14]{Ike19}}]\label{lem:ike-intersection}
    Let $N, L \subset T^*M$ be exact subsets and $\Lambda_N, \Lambda_L \subset T^*(M \times \bR_t)$ be their conifications. 
    Suppose $\Lambda_N$ and $T_{a}(\Lambda_L)$ intersect cleanly along $C$ and are smooth conical Lagrangians in a neighborhood of $C$. 
    Let $F_N \in \cT_{N}(T^*M)$ and $F_L \in \cT_{L}(T^*M)$ be simple sheaves along $\Lambda_N$ and $\Lambda_L$ in a neighborhood of $C$. 
    Then
    \begin{equation*}
        \Gamma_{[a,+\infty)}(\pi_*\cHom^\star(F_N, F_L))_a \simeq \Gamma(C; (\bF_2)_C)[-\mu(C)]
    \end{equation*}
    for some integer $\mu(C) \in \bZ$ determined by the Maslov index. 
\end{lemma}

We first consider the case where a cobordism has one Lagrangian for each positive and negative direction.

\begin{proposition}\label{prop:one_end}
    Let $V \subset T^*(M \times J)$ be an exact Lagrangian cobordism with conical ends between $L$ and $L'$. 
    Let $N$ be a Lagrangian and assume that the Lagrangians are so that $N$ intersects $L$ and $L'$ cleanly. 
    Then, there exists $\delta>0$ depending on $L'$ and $N$ such that 
    if $\cS(V) < \delta$, one has
    \begin{equation*}
        \sum_{n} \dim H^n(N \cap L';\bF_2) \ge \sum_{n} \dim H^n(N \cap L;\bF_2).
    \end{equation*}
    In particular, if $N$ intersects $L$ and $L'$ transversally, then there exists $\delta>0$ depending on $L'$ and $N$ such that if $\cS(V) < \delta$, one has     
    \begin{equation*}
        \#(N \cap L') \ge \#(N \cap L).
    \end{equation*}
\end{proposition}

In the proof, we consider $\cHom^\star(F,G)$ for simple sheaf quantizations $F,G \in \cT(T^*M)$.  
As mentioned before, we project $\cHom^\star$ to the left orthogonal ${}^\perp\Sh_{\{\tau \le 0\}}$.
If $\pi_* \cHom(F,G) \in {}^\perp\Sh_{\{\tau \le 0\}}(\bR) \simeq \cT(\pt)$ is constructible, it can be decomposed into interval modules as follows (see \cite[Cor.~7.3]{Gu16} and \cite[Thm.~1.7]{KS18persistent}):
\begin{equation*}
    \pi_* \cHom^\star(F,G) \simeq \bigoplus_{\alpha \in A} \bfk_{I_\alpha}[k_\alpha],
\end{equation*}
where $I_\alpha$ is an interval of the form $[a,a')$.
Each interval $I_\alpha$ is called a bar of $\pi_*\cHom^\star(F,G)$.
It is clear that 
$\Gamma_{[a,+\infty)}(\pi_* \cHom^\star(F, G))_a \neq 0$
only when $a$ is the end point of some bar $I_\alpha$.

\begin{proof}[Proof of \cref{prop:one_end}]
    Let $F_V \in \cT(T^*(M \times J))$ be a simple sheaf quantization of $V$ constructed in \cref{thm:sheaf-quan-cob-filling}. 
    Set $F_L \coloneqq (F_V)_{-1}$ and $F_{L'} \coloneqq (F_V)_{+1}$, which are simple sheaf quantizations of $L$ and $L'$.
    As already mention before, $\cS(F_V)$ coincides with the Lagrangian shadow $\cS(V)$ of $V$.
    Moreover, let $F_N$ be a simple sheaf quantization of $N$.
    
    We take primitives $g$ and $f$ of $N$ and $L$ respectively. Let
    \begin{equation*}
        \cA(p) \coloneqq g(p) - f(p), \; p \in N \cap L.
    \end{equation*}
    The end point of bars of $\pi_* \cHom^\star(F_N,F_{L})$ is in $\{ \cA(p) \mid p \in N \cap L \}$.
    Hence, we can consider
    \begin{equation*}
        \delta \coloneqq \min \{ \cA(p) - \cA(p') > 0 \mid p \neq p' \in N \cap L\},
    \end{equation*}
    being a lower bound on the length of all the bars of $\pi_* \cHom^\star(F_N,F_{L})$.
    By \cref{proposition:inequality_one_end},
    \begin{equation*}
        d'_{\cT(T^*M)}(F_L, T_cF_{L'})=d'_{\cT(T^*M)}(F_{-1}, T_c F_{+1}) \le \cS(F) < \delta.
    \end{equation*}
    Hence, for $\cS(F_V) = \delta'<\delta$, by \cref{def:defab}, we find that there exists $\delta'' \in [0, \delta']$ such that $\tau_{\delta'}(F_{L})$ factors as 
    \begin{equation*}
        \pi_* \cHom^\star(F_N,F_L) \to 
        \pi_* \cHom^\star(F_N,T_{c+\delta''}F_{L'}) \to 
        \pi_* \cHom^\star(F_N,T_{\delta'} F_L),
    \end{equation*}
    which implies that the number of bars of the mid term $\pi_*\cHom^\star(F_N, F_{L'})$ 
    is at least that of $\pi_* \cHom^\star(F_N,F_{L})$.
    By \cref{lem:ike-intersection}, the number of end points of bars of $\pi_* \cHom^\star(F_N,F_{L})$ is equal to $\sum_{n} \dim H^n(N \cap L;\bfk)$ and the number of end points of bars of $\pi_* \cHom^\star(F_N,F_{L'})$ is $\sum_{n} \dim H^n(N \cap L';\bfk)$. Hence we obtain the result.
\end{proof}

Now let us consider the case where a cobordism has multiple ends on the negative side. This is the most technical result in this paper and will occupy the rest of the paper. The main difficulty is to understand the relationship between the filtrations on $\bR$ and iterated cone decompositions in dg-categories, where we need to estimate not only the energy of the differential but also the energy of all higher homotopies.

\begin{theorem}[{cf.~\cite[Thm.~C]{BCS18}}]\label{theorem:intersection}
    Let $V \subset T^*(M \times J)$ be an exact Lagrangian cobordism with conical ends between $(L_0, \dots, L_{m-1})$ and $L'$. 
    Moreover, let $N$ be a Lagrangian and assume that $N$ intersects $L_i$ cleanly for $i=0,\dots,m-1$, $N$ intersects $L'$ cleanly, and there is no point in the intersection of the triple $(N,L_i,L_j)$ for $i \neq j$. 
    There exists a constant $\delta>0$ that depends only on $L_0, \dots, L_{m-1}$ and $N$ satisfying the following: if $\cS(V) < \delta$ then 
    \begin{equation*}
        \sum_{n \in \bZ} \dim H^n(N \cap L';\bF_2) \ge \sum_{i=0}^{m-1} \sum_{n \in \bZ} \dim H^n(N \cap L_i; \bF_2).
    \end{equation*}
\end{theorem}

\begin{corollary}
    In the situation of \cref{theorem:intersection}, assume furthermore that $N$ intersects $L_i$ transversally for $i=0,\dots,m-1$, $N$ intersects $L'$ transversally, and there is no point in the intersection of the triple $(N,L_i,L_j)$ for $i \neq j$. 
    Then, there exists a constant $\delta>0$ that depends only on $L_0, \dots, L_{m-1}$ and $N$ such that if $\cS(V) < \delta$, one has
    \begin{equation*}
        \#(N \cap L') \ge \sum_{i=0}^{m-1} \#(N \cap L_i).
    \end{equation*}
\end{corollary}

Let $g$ be the primitive of $N$, $f_i$ be the primitives of $L_i'$, and $\cA_i(p_i) \coloneqq g(p_i) - f_i(p_i),\, p_i \in N \cap L_i$. One can easily show that, by \cref{theorem:inequality_multiple_ends} and \cref{lem:ike-intersection}, we can take
\begin{equation}\label{eq:delta_0}
    \delta \coloneqq \min\{\cA_j(p_j) - \cA_i(p_i) \mid 0 \leq i \leq j \leq m-1, p_s \in N \cap L_s\}
\end{equation}
to ensure that the theorem holds. However, this bound $\delta > 0$ depends not only on the underlying Lagrangian submanifolds $N, L_0, \dots, L_{m-1}$, but on the primitives $f_0, \dots, f_{m-1}$ and $c_1, \dots, c_{m-1}$ as well. Therefore, we will find a lower bound for the lengths of the bars of
\begin{equation}
    \pi_* \cHom^\star(F_{N}, [F_{L_0} \to T_{c_1}F_{L_1} \to \dots \to T_{c_{m-1}}F_{L_{m-1}}]),
\end{equation}
among all possible shifts that give different choices of primitives.

Before going into the proof, we fix our convention and terminology.
\begin{enumerate}
    \item[(a)] For a chain map between chain complexes $\varphi = (\varphi^n)_n \colon X_0 \to X_1$, the mapping cone $\Cone(\varphi)$ is defined as 
    \begin{equation}\label{eqn:cone-definition}
        \Cone(\varphi)^n = X_1^{n} \oplus X_0^{n+1}, 
        \quad 
        d_{\Cone(\varphi)}^n = 
        \mqty[d_{X_1}^{n} & \varphi^n \\ 0 & d_{X_0[1]}^{n}],
    \end{equation}
    where $d_{X[1]}^n=-d_X^{n+1}$. Thus, an iterated mapping cone or a twisted complex \cite{BondalKapranov} $[X_0 \to X_1 \to \dots \to X_{m-1}]$ can be written as a complex
    \begin{equation}
    [X_0 \to X_1 \to \dots \to X_{m-1}]^n = X_{m-1}^n \oplus \dots \oplus X_1^{n+m-2} \oplus X_0^{n+m-1},
    \end{equation}
    with differential defined as
    \begin{equation}\label{eqn:twist-complex-definition}
    \mqty[
        d_{m-1} & \varphi_{m-2,m-1} & \varphi_{m-3,m-1} & \cdots & \cdots & \varphi_{0,m-1} \\
        0 & d_{m-2}[1] & \varphi_{m-3,m-2}[1] &  \cdots & \cdots & \varphi_{0,m-2}[1] \\
        \vdots & 0 & d_{m-3}[2] & \cdots & \cdots & \varphi_{0,m-3}[2] \\  
        \vdots & \vdots & \ddots & \ddots & \vdots & \vdots \\ 
        0 & 0 & \cdots & 0 &  d_1[m-2] & \varphi_{0,1}[m-2] \\
        0 & 0 & \cdots & \cdots & 0 & d_{0}[m-1]
    ].
    \end{equation}
    They satisfy the relation
    \begin{equation}\label{eqn:twist-complex-null-homotopy}
        \varphi_{i,j} d_i + \sum_{i < k < j} \varphi_{k,j} \varphi_{i,k} = d_j \varphi_{i,j}
    \end{equation}
    for any $1 \le i < j \le m-1$.
    This means that $\sum_{i < k < j} \varphi_{k,j} \varphi_{i,k}$ is a chain map $X_i \to X_j[i-j+1]$ and $\varphi_{i,j}$ is a null homotopy of the chain map. 
    \item[(b)] Let $\varphi \colon H_0 \to H_1$ be a morphism of constructible objects in ${}^\perp \Sh_{\{\tau \le 0\}}(\bR) \simeq \cT(\pt)$.
    First we can decompose each $H_i$ into interval modules, and then describe them as complexes of the direct sums of shifted intervals of the form $\bfk_{[a,\infty)}$:
    \begin{equation}\label{eqn:decompose-infinite-interval}
        H_0 \simeq \lb \bigoplus_{p} \bfk_{[a_{p},\infty)}[k_p], d_0 \rb, \quad 
        H_1 \simeq \lb \bigoplus_{q} \bfk_{[a_{q},\infty)}[l_q], d_1 \rb. 
    \end{equation}
    Assume both of the direct sums in \eqref{eqn:decompose-infinite-interval} are finite.
    In this case, the morphism $\varphi$ can be represented as a chain map between the above complexes of sheaves, whose derivatives are compatible with the interval decompositions. 
    One can check this by computing the morphisms between interval modules (see, for example, \cite{berkouk2021derived}).
    We say that the energy of $d_i$ is at least $e>0$ if every non-trivial component of $d_i$ comes from a morphism $\bfk_{[a,\infty)} \to \bfk_{[a',\infty)}$ with $a'-a \ge e$.
    Similarly, we say that the energy of the representing chain map $\varphi$ is at least $e\geq0$ if every non-trivial component of $\varphi$ comes from a morphism $\bfk_{[a,\infty)} \to \bfk_{[a',\infty)}$ with $a'-a \ge e$.
    By \eqref{eqn:cone-definition}, we see that if the energy of $d_0$, $d_1$, and $\varphi$ are at least $e>0$, then the length of the shortest bar of the mapping cone $\Cone(\varphi)$ is at least $e$. 
    
    \item[(c)] For a morphism $\psi \colon G \to H$ and $b \ge 0$, we say a morphism $\tilde{\psi}_b \colon G \to T_{-b}H$ is a \emph{lift} of $\psi \colon G \to H$ if the following diagram commutes up to homotopy:
    \begin{equation}
    \begin{aligned}
    \xymatrix{
        & T_{-b}H \ar[d]^-{\tau_{-b,0}(H)} \\
        G \ar[r]_-{\psi} \ar[ru]^-{\tilde{\psi}_b} & H.
    }
    \end{aligned}
    \end{equation}
    We see that if $\varphi$ admits a lift $\tilde{\varphi}_{b-\varepsilon} \colon G \to T_{-b+\varepsilon}H$ for any $\varepsilon > 0$ sufficiently small, then $\varphi$ must also admit a lift $\tilde{\varphi}_{b} \colon G \to T_{-b}H$. In fact, when $\MS(G)$ and $\MS(H)$ are conical Lagrangians in a neighborhood of $\MS(G) \cap T_{-b}\MS(H)$, $\tau_{-b, -b+\varepsilon}(H)$ induces a natural isomorphism
    \begin{equation}
        \Hom(G, T_{-b}H) \xrightarrow{\sim} \Hom(G, T_{-b+\varepsilon}H)
    \end{equation}
    for sufficiently small $\varepsilon > 0$.
\end{enumerate}

First, we prove the case $m=2$. Since there are two Lagrangians on the end of the cobordism, by \cref{theorem:iterated_cone}, we need to estimate the energy of the morphisms involved in a mapping cone
\begin{equation*}
    \Cone(\pi_*\cHom^\star(F_N, G_0) \to \pi_*\cHom^\star(F_N, G_1)).
\end{equation*}
For that purpose, we need the following lemma.

\begin{lemma}\label{lem:energy-two-ends}
    Let $N, L_0, L_1 \subset T^*M$ be exact Lagrangians with primitives $g, g_0, g_1$. Let $F_N \in \cT_{N}(T^*M)$, $G_0 \in \cT_{L_0}(T^*M)$, and $G_1 \in \cT_{L_1}(T^*M)$ be simple sheaf quantizations of $N$, $L_0$, and $L_1$. 
    Then for any morphism $\varphi_{0,1} \colon G_0 \to G_1$, the energy of the induced morphism ${\varphi_{0,1}}_* \colon \pi_*\cHom^\star(F_N, G_0) \to \pi_*\cHom^\star(F_N, G_1)$
    is at least
    \begin{equation*}
        \min \{ \cA_1(q) -\cA_0(p) -\cA_{0,1}(r) > 0 \mid p \in N \cap L_0, q \in N \cap L_1, r \in L_0 \cap L_1 \},  
    \end{equation*}
    where for $i = 0, 1$,
    \begin{align*}
        \cA_i(p) & \coloneqq g(p) -g_i(p) \quad (p \in N \cap L_i), \\
        \cA_{0,1}(r) & \coloneqq g_0(r) - g_1(r) \quad (r \in L_0 \cap L_1).
    \end{align*}
\end{lemma}
\begin{proof}
    Using \cref{lem:ike-intersection}, we know an end point of a bar of $\pi_* \cHom^\star(F_N,G_0)$ is in the set $\{ \cA_0(p) \mid p \in N \cap L_0 \}$, and an end point of a bar of $\pi_*\cHom^\star(F_N,G_1)$ is in the set $\{ \cA_1(q) \mid q \in N \cap L_1 \}$. We pick a basis $B(N \cap L_i)$ for $H_*(N \cap L_i; \bfk)$. Note that on every connected component of $N \cap L_i$, the action $\cA_i(p)$ is a constant, so $\cA_i$ is well defined on $B(N \cap L_i)$.
    Following \eqref{eqn:decompose-infinite-interval}, we can rewrite
    the interval decompositions of $\pi_* \cHom^\star(F_N,G_0)$ and $\pi_* \cHom^\star(F_N,G_1)$ as complexes of the direct sums of intervals of the form $\bfk_{[a,\infty)}$: 
    \begin{align*}
        \pi_*\cHom^\star(F_N,G_0) 
        & \simeq \lb \bigoplus_{p \in B(N \cap L_0)} \bfk_{[\cA_0({p}),\infty)}[k_{p}], {d_0}_* \rb, \\
        \pi_*\cHom^\star(F_N,G_1) 
        & \simeq \lb \bigoplus_{q \in B(N \cap L_1)} \bfk_{[\cA_1({q}),\infty)}[l_{q}], {d_1}_* \rb,
    \end{align*}
    where the identifications mean that both sides can be connected by a zig-zag of chain maps that are equivalences. 
    Then, fix a chain map
    \begin{equation*}
        {\varphi_{0,1}}_* \colon \lb \bigoplus_{p} \bfk_{[\cA_0({p}),\infty)}[k_p], {d_0}_* \rb \to \lb \bigoplus_{q} \bfk_{[\cA_1({q}),\infty)}[l_q], {d_1}_* \rb
    \end{equation*}
    that coincides with the the morphism $\pi_*\cHom^\star(F_N,G_0)\to \pi_*\cHom^\star(F_N,G_1)$ induced by $\varphi_{0,1} \colon G_0 \to G_1$ through the above identifications up to homotopy.
    
    We follow the terminology in (c) above. 
    By the remark there, we can set 
    \begin{equation*}
        b_{0,1} \coloneqq \max \lc b \in \bR_{\ge 0} \relmid \text{there is a lift $\tilde{\varphi}_{0,1; b} \colon G_0 \to T_{-b}G_1$ of $\varphi_{0,1} \colon G_0 \to G_1$} \rc.
    \end{equation*}
    First, every non-trivial component of ${\varphi_{0,1}}_*$ is of the form $\bfk_{[a,\infty)} \to \bfk_{[a',\infty)}$ with 
    \begin{equation}\label{eq:morphism_energy}
        a' -a \ge \min \{\cA_1(q) -\cA_0(p) - b_{0,1} \ge 0 \mid p \in N \cap L_0, q \in N \cap L_1 \}.
    \end{equation}    
    Since we cannot extend a lift $\tilde{\varphi}_{0,1; b}$ beyond $b_{0,1}$, we know that $H^0(\Hom(G_0, T_{-b_{0,1}-\varepsilon}G_1) )\to H^0(\Hom(G_0, T_{-b_{0,1}}G_1))$ is not an isomorphism for sufficiently small $\varepsilon > 0$. 
    Note that
    \begin{equation*}
        \Gamma_{[a,+\infty)}(\pi_* \cHom^\star(G_0, G_1))_{a} = \Cone(\Hom(G_0, T_{a-\varepsilon}G_1) \to \Hom(G_0, T_{a}G_1)).
    \end{equation*}
    Hence we can apply \cref{lem:ike-intersection} and get
    \begin{equation*}
        b_{0,1} \in \{ \cA_{0,1}(r) \mid r \in L_0 \cap L_1 \}.
    \end{equation*}
    The right-hand side of \eqref{eq:morphism_energy} is thus bounded below by 
    \begin{equation*}
        \min \{ \cA_1(q) -\cA_0(p) -\cA_{0,1}(r) \ge 0 \mid p \in N \cap L_0, q \in N \cap L_1, r \in L_0 \cap L_1 \},
    \end{equation*}
    where the case $\cA_1(q) -\cA_0(p) -\cA_{0,1}(r) = 0 $ will result in cancellations of endpoints.
    
    Then we show that the morphism induced by the lift 
    \[
    \begin{tikzcd}
        \pi_*\cHom^\star(F_N,G_0) \ar[r] \ar[d, "\sim" {anchor=south, rotate=90}]
        & \pi_*\cHom^\star(F_N,T_{-b_{0,1}}G_1) \ar[d, "\sim" {anchor=north, rotate=90}] \\
        \lb \bigoplus_{p} \bfk_{[\cA(p),\infty)}[k_p], {d_0}_* \rb \ar[r] 
        & \lb \bigoplus_{q} \bfk_{[\cA(q)-b_{0,1},\infty)}[l_q], {d'_1}_* \rb
    \end{tikzcd}
    \]
    does not make any cancellations of endpoints. Consider \cref{lem:ike-intersection} and let $\Lambda_N = \Lambda_N$ and $\Lambda_L = \Lambda_{L_0} \cup T_{-b_{0,1}}(\Lambda_{L_1})$. Since $N \cap L_0 \cap L_1 = \varnothing$, we know that $\Lambda_N$ and $T_{a}(\Lambda_L)$ intersect cleanly for any $a \in \bR$, and $\Cone(G_0 \to T_{-b_{0,1}}G_1)$ is simple along the smooth strata of $\Lambda_L$. Therefore, by \cref{lem:ike-intersection}, the number of endpoints of bars of 
    \begin{equation*}
        \pi_*\cHom^\star(F_N, \Cone(G_0 \to T_{-b_{0,1}}G_1))
    \end{equation*}
    is $\sum_{n} \dim H^n(N \cap L_0;\bfk) + \sum_{n} \dim H^n(N \cap L_1;\bfk)$, and we conclude that there are no cancellations of endpoints.  
    Hence, the right-hand side of \eqref{eq:morphism_energy} is bounded below by 
    \begin{equation*}
        \min \{ \cA_1(q) -\cA_0(p) -\cA_{0,1}(r) > 0 \mid p \in N \cap L_0, q \in N \cap L_1, r \in L_0 \cap L_1 \},
    \end{equation*}
    which proves the lemma.
\end{proof}

\begin{proof}[Proof of \cref{theorem:intersection}: the case $m=2$]
    Let $F_V \in \cT(T^*(M \times J))$ be a simple sheaf quantization of $V$ constructed in \cref{thm:sheaf-quan-cob-filling}. Set each $F_{L_i} \coloneqq (F_V)^i_{-1}$ in the iterated cone decomposition in \cref{theorem:iterated_cone}, which is a simple sheaf quantization of $L_i$ for some primitives $f_i$.
    We also set $F_{L'} \coloneqq (F_V)_{+1}$, which is a simple sheaf quantization of $L'$. 
    The shadow $\cS(F_V)$ coincides with the Lagrangian shadow $\cS(V)$.
    Moreover, let $F_N$ be a simple sheaf quantization of $N$ with primitive $g$.

    By \cref{theorem:inequality_multiple_ends}, we get
    \begin{equation*}
        d'_{\cT(T^*M)}(\Cone(F_{L_0}[-1] \to T_{c_1}F_{L_1}),F_{L'}) = d'_{\cT(T^*M)}(\tilde{F}_{-1}, F_{+1}) \le \cS(F).
    \end{equation*}
    We shall prove that the length of the shortest bar of 
    \begin{equation*}
        \pi_* \cHom^\star(F_N, \Cone(F_{L_0}[-1] \to T_{c_1}F_{L_1}))
    \end{equation*}
    is bounded below by a constant that depends only on $N$, $L_0$, and $L_1$. Set 
    \begin{equation*}
        G_0 \coloneqq T_{c_0}F_{L_0}[-1], \quad G_1 \coloneqq T_{c_1}F_{L_1}
    \end{equation*}
    and $g_i \coloneqq f_i-c_i$ for $i=0,1$. 
    By \eqref{eqn:cone-definition} the cone of $\varphi_{0,1}$ is isomorphic to the complex 
    \begin{equation*}
        \lb \bigoplus_{q} \bfk_{[\cA_1({q}),\infty)}[l_q] \oplus \bigoplus_{p} \bfk_{[\cA_0({p}),\infty)}[k_p+1], \mqty[{d_1}_* & {\varphi_{0,1}}_* \\ 0 & -{d_0}_*[1]] \rb. 
    \end{equation*}
    The energy of the differential ${d_0}_*$ (resp.\ ${d_1}_*$) is at least 
    \begin{align*}
        & \min \{\cA_0(p) -\cA_0(p') >0 \mid p, p' \in N \cap L_0 \},\\
        \text{(resp. }& \min \{ \cA_1(q) - \cA_1(q')>0 \mid q, q' \in N \cap L_1 \}).
    \end{align*}    
    Moreover, by \cref{lem:energy-two-ends}, the energy of ${\varphi_{0,1}}_*$ is at least
    \begin{equation*}
        \min \{ \cA_1(q) -\cA_0(p) -\cA_{0,1}(r) > 0 \mid p \in N \cap L_0, q \in N \cap L_1, r \in L_0 \cap L_1 \}.
    \end{equation*} 
    Combining the above inequalities, we conclude that the length of any bar of the concerned sheaf $\pi_* \cHom^\star(F_N, \Cone(G_0 \to G_1))$ is at least
    \begin{equation*}
        \delta \coloneqq \min \lc
        \begin{aligned}
            &\cA_0(p) -\cA_0(p') >0,\ \cA_1(q) - \cA_1(q')>0, \\
            &\cA_1(q) -\cA_0(p) -\cA_{0,1}(r) > 0 
        \end{aligned}
        \relmid
        \begin{aligned}
            &p, p' \in N \cap L_0,\\
            &q, q' \in N \cap L_1, \\
            &r \in L_0 \cap L_1.
        \end{aligned}\rc.
    \end{equation*}
    This quantity is independent of the shifts $c_0$, $c_1$, and the primitives $g, f_0, f_1$. Then, when $\cS(F_V) <\delta$, 
    the number of end points of bars of $\pi_*\cHom^\star(F_N, F_{L'})$ is at least that of 
    $\pi_* \cHom^\star(F_N,\Cone(G_0 \to G_1))$.
    Hence by \cref{lem:ike-intersection}, $\sum_{n} \dim H^n(N \cap L';\bfk)$ is at least $\sum_{n} \dim H^n(N \cap L_0;\bfk) + \sum_{n} \dim H^n(N \cap L_1;\bfk)$. 
    This completes the proof in the case $m=2$.
\end{proof}

For the proof of \cref{theorem:intersection} in the general case, we need to generalize \cref{lem:energy-two-ends} to the case $m \ge 3$ and estimate the energy of the morphisms in an iterated cone decomposition, or equivalently, a twisted complex \cite{BondalKapranov}. 
To this end, we consider the following situation.
Let $G_i \in \cT(T^*M) \ (i=0,\dots,m-1)$. Similar to \eqref{eqn:twist-complex-definition}, a twisted complex
\begin{equation*}
    [G_0 \to G_1 \to \dots \to G_{m-2} \to G_{m-1}]
\end{equation*}
is described up to sign as $G_{m-1} \oplus \dots \oplus G_{1}[2-m] \oplus G_{0}[1-m]$ with an upper triangular differential $[\varphi_{m-1-i,m-1-j}]_{i<j}$, where $\varphi_{i,j} \colon G_i \to G_j$ is a non-closed morphism of degree $i-j+1$ that defines a null homotopy as explained in \eqref{eqn:twist-complex-null-homotopy}. For $F \in \cT(T^*M)$, it induces a twisted complex
\begin{equation}\label{eqn:iterated-cone}
    [\pi_*\cHom^\star(F, G_0) \to \pi_*\cHom^\star(F, G_1) \to \dots \to \pi_*\cHom^\star(F, G_{m-1})].
\end{equation}
We denote by ${\varphi_{i,j}}_* \colon \pi_* \cHom^\star(F_{N}, G_i) \to \pi_*\cHom^\star(F_{N}, G_j)$ the morphism induced by $\varphi_{i,j} \colon G_i \to G_j$ and choose a chain homotopy between these sheaves on $\bR$ that represents ${\varphi_{i,j}}_*$.

\begin{lemma}\label{lemma:energy_general}
    Let $N, L_0, \dots, L_{m-1} \subset T^*M$ be exact Lagrangians with primitives $g, g_0, \dots$, $g_{m-1}$. Let $F_N \in \cT_{N}(T^*M)$ and $G_i \in \cT_{L_i}(T^*M)$ be simple sheaf quantizations of $N$ and $L_i$. Consider the twisted complex 
    \begin{equation*}
        [G_0 \to G_1 \to \dots \to G_{m-2} \to G_{m-1}]
    \end{equation*}
    with morphisms $\varphi_{i,j} \colon G_i \to G_j$ of degree $i-j+1$. Then for any $i < j$, the energy of the induced morphism ${\varphi_{i,j}}_* \colon \pi_* \cHom^\star(F_{N}, G_i) \to \pi_*\cHom^\star(F_{N}, G_j)$ is at least 
    \begin{equation}\label{eq:energy_lower_bound_ij}
        \min \lc
        \begin{aligned}
            & \cA_{j}(p_{j}) - \cA_{i}(p_{i}) - \cA_{i,i_1}(p_{i,i_1}) \\
            & - \cA_{i_1,i_2}(p_{i_1,i_2}) - \dots - \cA_{i_k,j}(p_{i_k,j}) > 0
        \end{aligned}
        \relmid 
        \begin{aligned}
            & p_s \in N \cap L_s \ (s=i,j) \\
            & i < i_1 < \dots < i_k < j \ (k \ge 0) \\
            & p_{s,t} \in L_s \cap L_t
        \end{aligned}
        \rc,
    \end{equation}
    where 
    \begin{align*}
        \cA_i(p_{i}) & \coloneqq g(p_{i})-g_i(p_{i}) \quad (p_i \in N \cap L_i), \\
        \cA_{i,j}(p_{i,j}) & \coloneqq g_i(p_{i,j}) - g_j(p_{i,j}) \quad (p_{i,j} \in L_i \cap L_j).
    \end{align*}
    When $k=0$, the value in \eqref{eq:energy_lower_bound_ij} should be understood as $\cA_{j}(p_{j}) - \cA_{i}(p_{i}) -\cA_{i,j}(p_{i,j})$.
\end{lemma}

Assuming the lemma, we can immediately finish the proof of the main theorem:

\begin{proof}[Proof of \cref{theorem:intersection}: the case $m \ge 3$]
    Let $F_V \in \cT(T^*(M \times J))$ be a simple sheaf quantization of $V$ constructed in \cref{thm:sheaf-quan-cob-filling}. 
    We set $F_{L_i} \coloneqq (F_V)^i_{-1}$, which is a simple sheaf quantization of $L_i$ for some primitive $f_i$ and set $F_{L'} \coloneqq (F_V)_{+1}$, which is a simple sheaf quantization of $L'$ with primitive $g$.
    Moreover, let $F_N$ be a simple sheaf quantization of $N$.
    We set $G_i \coloneqq T_{c_i}F_i[m-i-1] \ (i=0,1,\dots,m-1)$, and the primitives $g_i \coloneqq f_i-c_i \ (i=0,1,\dots,m-1)$. 

    We denote by ${d_i}_*$ the associated differential of $\pi_* \cHom^\star(F_N,G_i)$.  
    Similar to the case $m=2$, the energy of ${d_i}_*$ is at least $\min \{\cA_i(p_i) -\cA_i(p'_i) >0 \mid p_i, p'_i \in N \cap L_i \}$ ($i=0,1,\dots,m-1$).
    Combining this with \cref{lemma:energy_general}, we find that for any choice of the primitives $f_i$ and shifts $c_i$ $(0 \leq i \leq m-1)$, the energy of $d_i \ (i=0,1,\dots,m-1)$ and ${\varphi_{i,j}}_* \ (0 \le i < j \le m-1)$ is at least
    \begin{equation}\label{eq:energy_lower_bound_uniform}
        \delta \coloneqq \min \lc
        \begin{aligned}
            & \cA_{j}(p_{j}) - \cA_{i}(p_{i}) - \cA_{i,i_1}(p_{i,i_1}) \\
            & - \cA_{i_1,i_2}(p_{i_1,i_2}) - \dots - \cA_{i_k,j}(p_{i_k,j}) > 0
        \end{aligned}
        \relmid 
        \begin{aligned}
            & 0 \le i \le j \le m-1 \\
            & p_s \in N \cap L_s \ (s=i,j) \\
            & i < i_1 < \dots < i_k < j \ (k \ge 0) \\
            & p_{s,t} \in L_s \cap L_t
        \end{aligned}
        \rc.
    \end{equation}
    Here, when $i=j$ and $k=0$, the value in the set should be understood as $\cA_{i}(p_{i})-\cA_{i}(p'_{i})$ with $p_{i},p'_{i} \in N \cap L_{i}$. 
    Hence, for $\cS(V) < \delta$, by the same arguments in \cref{proposition:inequality_one_end}, the number of endpoints of bars $\pi_*\cHom^\star(F_N, F_{L'})$, which by \cref{lem:ike-intersection} is equal to $\sum_{n} \dim H^n(N \cap L';\bfk)$, is at least that of 
    \begin{equation*}
        \pi_*\cHom^\star(F_N, [G_0 \to T_{c_1}G_1 \to \dots \to T_{c_{m-1}}G_{m-1}]),
    \end{equation*}
    which by \cref{lem:ike-intersection} is equal to $\sum_{i=0}^{m-1} \sum_{n} \dim H^n(N \cap L_i;\bfk)$.
    This completes the proof of \cref{theorem:intersection}. 
\end{proof}   

The rest of this section is occupied by the proof of \cref{lemma:energy_general}. 
First, rather than estimating the energy of the null homotopy ${\varphi_{i,j}}_*$, we choose to estimate the energy of some corresponding chain map.
We observe the following simple lemma, which implies that the null homotopy $\varphi_{i,j}$ in \eqref{eqn:twist-complex-null-homotopy} corresponds to a chain map $G_i[1] \to \Cone(\sum_{i < k < j} \varphi_{k,j} \varphi_{i,k})$.
    
\begin{lemma}\label{lemma:mapping_cone}
    Let $\psi \colon X \to Y$ be a chain map between chain complexes.
    Then there is a one-to-one correspondence between null homotopies of $\psi$ and chain maps $X[1] \to \Cone(\psi)$ such that $X[1] \to \Cone(\psi) \to X[1]$ is equal to the identity.
\end{lemma}
\begin{proof}
    Let $(s^n \colon X^{n} \to Y^{n-1})_n$ be a null homotopy of $\psi$ so that they satisfy $\psi^n=d_Y^{n-1}s^n+s^{n+1}d_X^n$.
    Set 
    \begin{equation}\label{eq:nullhtpy-isom}
        \eta^n \coloneqq \mqty[-s^{n+1} \\ \id_{X^{n+1}}] \colon X^{n+1} \to Y^n \oplus X^{n+1}.
    \end{equation} 
    Then, we have
    \begin{equation*}
        \mqty[d_Y^{n} & \psi^{n+1} \\ 0 & d_{X[1]}^n] \mqty[-s^{n+1} \\ \id_{X^{n+1}}] 
        =
        \mqty[-d_Y^{n} s^{n+1} + \psi^{n+1} \\ d_{X[1]}^n]
        = 
        \mqty[-s^{n+2} \\ \id_{X^{n+2}}] d_{X[1]}^n,
    \end{equation*}
    which shows that $\eta=(\eta^n)_n \colon X[1] \to \Cone(\psi)$ is a chain map such that the composite $X[1] \xrightarrow{\eta} \Cone(\psi) \to X[1]$ is the identity.

    Conversely, $\eta=(\eta^n)_n \colon X[1] \to \Cone(\psi)$ is a chain map such that $X[1] \xrightarrow{\eta} \Cone(\psi) \to X[1]$ is the identity.
    We can find that $\eta_n$ has to be of the form \eqref{eq:nullhtpy-isom}. 
    By the above computation, the commutativity $d_{\Cone(\psi)}^n \eta^n = \eta^{n+1} d_{X[1]}^n$ implies that $-d_Y^{n+1}s^{n+1} + \psi^{n+1} = s^{n+2}d_{X}^{n+1}$.
    This shows that $(s^n)_n$ is a null homotopy of $\psi$.
\end{proof}

We will now prove \cref{lemma:energy_general}. \cref{lem:energy-two-ends} estimates the energy of the chain map ${\varphi_{i,j}}_*$ when $j-i = 1$. In general, we will need to estimate the energy of ${\varphi_{i,j}}_*$ using induction on $j-i$ by proving the following claim, which immediately implies \cref{lemma:energy_general}.
    
\begin{claim}\label{claim:general}
    For any $i<j$, one can define $b_{i,j} \in \bR_{\ge 0}$ and a degree $i-j+1$ map $\tilde{\varphi}_{i,j;b_{i,j}} \colon G_i \to T_{-b_{i,j}}G_j$ that satisfy the following properties:
    \begin{enumerate}
        \item \label{claim:first}
        \begin{equation}\label{eq:bij_set}
            b_{i,j} \in 
            \lc 
                \cA_{i,i_1}(p_{0,i_1}) + \cA_{i_1,i_2}(p_{i_1,i_2}) + \dots + \cA_{i_k,j}(p_{i_k,j})                
            \relmid
            \begin{aligned}
                & i < i_1 < \dots < i_k < j \ (k \ge 0)\\
                & p_{s,t} \in L_s \cap L_t
            \end{aligned}
        \rc,
        \end{equation} 
        \item the twisted complex formed by the iterated mapping cone \eqref{eqn:iterated-cone} can be lifted to a twisted complex formed by the iterated mapping cone
        \begin{equation*}
           \big[ \pi_*\cHom^\star(F_N, G_{i})
            \to \pi_*\cHom^\star(F_N, T_{-b_{i,i+1}}G_{i+1}) \to \dots \to \pi_*\cHom^\star(F_N, T_{-b_{i,j}}G_{j}) \big]
        \end{equation*} 
        with the differential given by lifts $\tilde{\varphi}_{s, t; b_{i,t}-b_{i,s}} \colon G_s \to T_{-b_{s,t}}G_t[s-t+1]$ of $\varphi_{s,t} \colon G_s \to G_t[s - t + 1]$ for $i \le s \le t \le j$,         
        and \label{claim:second}
        \item the energy of ${\varphi_{i,j}}_*$ is at least \label{claim:third}
        \begin{equation*}
            \min \lc \cA_j(p_j) - \cA_i(p_i) - b_{i,j} > 0 \relmid p_s \in N \cap L_s \rc.
        \end{equation*}
    \end{enumerate}
    Here, when $k=0$, the value in \eqref{eq:bij_set} should be understood as $\cA_{i,j}(p_{i,j})$.           
\end{claim}

\begin{proof}
    \textbf{(I)} For $j-i=1$, we define
    \begin{equation}
        b_{i,j} \coloneqq \max \lc b \in \bR_{\ge 0} \relmid \text{there exists a lift $\tilde{\varphi}_{i,j;b} \colon G_i \to T_{-b}G_{j}$ of $\varphi_{i,j} \colon G_i \to G_{j}$} \rc.
    \end{equation}
    By the proof of \cref{lem:energy-two-ends} (for the case $m=2$), we have proved \cref{claim:general}.
    Now we fix a maximal lift $\tilde{\varphi}_{i,j;b_{i,j}} \colon G_i \to T_{-b_{i,j}}G_{j}$ for $j-i=1$.
    \smallskip

    \textbf{(II)} Suppose that for any $i,j$ with $j-i<l$, we have defined $b_{i,j}$, fixed a maximal lift $\tilde{\varphi}_{i,j;b_{i,j}} \colon G_i \to T_{-b_{i,j}}G_j$, and proved \cref{claim:general}. Now we will prove \cref{claim:general} for $j-i=l$.
    \smallskip
    
    \textbf{Step 1: Definition of $b_{i,j}$}. We define $b_{i,j}$ for $j-i=l$ as follows.
    Set 
    \begin{equation}\label{eqn:definition_tildeb}
        \tilde{b}_{i,j} \coloneqq \min\{ b_{i,k} + b_{k,j} \mid i<k<j \}
    \end{equation}
    and consider the lift of $\psi_{i,j} \coloneqq \sum_{i < k < j} \varphi_{k,j} \varphi_{i,k}$ of the form
    \begin{equation*}
        \tilde{\psi}_{i,j;b} \coloneqq \sum_{i < k < j} \tau_{-b_{i,k} - b_{k,j},-b}(G_j) \tilde{\varphi}_{k,j;b_{k,j}} \tilde{\varphi}_{i,k;b_{i,k}} 
    \end{equation*}
    for $0 \le b \le \tilde{b}_{i,j}$.
    Note that by \eqref{eqn:twist-complex-null-homotopy}, $\varphi_{i,j}$ is a null homotopy of $\psi_{i,j}$. We then define
    \begin{equation}\label{eq:def-induction-b_ij}
        b_{i,j} \coloneqq \lc b \in [0,\tilde{b}_{i,j}] \relmid 
        \begin{aligned}
            & \text{there exists a null homotopy $\tilde{\varphi}_{i,j;b}$} \\
            & \text{of $\psi_{i,j;b}$ that is a lift of $\varphi_{i,j}$}
        \end{aligned}
         \rc.
    \end{equation}
    Here the condition means that the following diagram commutes when $\varphi_{i,j}$ and $\tilde{\varphi}_{i,j;b}$ are regarded as chain maps by \cref{lemma:mapping_cone}:
    \[
    \begin{tikzcd}
        & \Cone(\tilde{\psi}_{i,j;b}) \ar[d] \\
        G_i[1] \ar[r, "{\varphi_{i,j}}"'] \ar[ru, "{\tilde{\varphi}_{i,j;b}}"] & \Cone(\psi_{i,j}). 
    \end{tikzcd}
    \]
    We claim that the quantity $b_{i,j}$ and the lift $\varphi_{i,j;b_{i.j}}$ defined above satisfy \cref{claim:general}.
    \smallskip
    
    \textbf{Step 2: Alternative characterization of $b_{i,j}$}. Let us investigate the condition for the existence of such a lift and provide an alternative characterization of $b_{i,j}$ in terms of morphisms from $G_i$ to $G_j$, which can then be used to prove the first assertion in \cref{claim:general}.
    For $b \le \tilde{b}_{i,j}$, we consider the following commutative diagram with solid arrows:
    \[
    \begin{tikzcd}
        G_i \ar[r] \ar[d] & T_{-b}G_j \ar[r] \ar[d] & \Cone(\tilde{\psi}_{i,j;b}) \ar[r] \ar[d,dotted] & G_i[1] \ar[d] \\
        G_i \ar[r] \ar[d] & G_j \ar[r] \ar[d] & \Cone(\psi_{i,j}) \ar[r] \ar[d,dotted] & G_i[1] \ar[d] \\
        0 \ar[r,dotted] \ar[d] & \bfk_{M \times [-b,0)} \star G_j[1] \ar[r,dotted] \ar[d] & H \ar[r,dotted] \ar[d,dotted] & 0 \ar[d] \\
        G_i[1] \ar[r] & T_{-b}G_j[1] \ar[r] & \Cone(\tilde{\psi}_{i,j;b})[1] \ar[r] & G_i[2]. 
    \end{tikzcd}
    \]
    Then by \cite[Exercise~10.6]{KS06} in the triangulated setting \cite[Lem.~3.6 Example~(3)]{AnnoLogvinenko} in the dg enhanced setting, the dotted arrows can be completed so that the right bottom square is anti-commutative, all the other squares are commutative, and all the rows and all the columns are exact triangles. 
    Hence we have $H \simeq \bfk_{M \times [-b,0)} \star G_j[1]$.
    The third column and $\varphi_{i,j} \colon G_i[1] \to \Cone(\psi)$ fit into the following diagram with solid arrows: 
    \[
    \begin{tikzcd}
        \Cone(\tilde{\psi}_{i,j;b}) \ar[d] & \\
        \Cone(\psi_{i,j}) \ar[d] & G_i[1] \ar[l,"{\varphi_{i,j}}"] \ar[lu,dotted] \\
        \bfk_{M \times [-b,0)} \star G_j[1] \ar[d] \\
        \Cone(\tilde{\psi}_{i,j;b})[1] & 
    \end{tikzcd}
    \]
    Hence, a lift $G_i[1] \to \Cone(\tilde{\psi}_{i,j;b})$ exists if and only if the composite $G_i[1] \xrightarrow{\varphi_{i,j}} \Cone(\psi_{i,j}) \to \bfk_{M \times [-b,0)} \star G_j[1]$ is zero. 
    Moreover, since $\varphi_{i,j}$ is a null homotopy, such a lift $\tilde{\varphi}_{i,j;b}$ automatically satisfies the condition that $G_i[1] \xrightarrow{\tilde{\varphi}_{i,j;b}} \Cone(\tilde{\psi}_{i,j;b}) \to G_i[1]$ is the identity by the following commutative diagram: 
    \[
    \begin{tikzcd}
        G_i[1] \ar[d, equal] & \Cone(\tilde{\psi}_{i,j;b}) \ar[d] \ar[l] & \\
        G_i[1] & \Cone(\psi_{i,j}) \ar[l] & G_i[1]. \ar[l, "{\varphi_{i,j}}"] \ar[lu, "{\tilde{\varphi}_{i,j;b}}"'] \ar[ll, bend left, "\id"']
    \end{tikzcd}
    \]
    As a result, we find that 
    \begin{equation}\label{eqn:characterization-of-bij}
        b_{i,j} = \max \lc b \in [0,\tilde{b}_{i,j}] \relmid \text{$G_i[1] \xrightarrow{\varphi_{i,j}} \Cone(\psi_{i,j}) \to \bfk_{M \times [-b,0)} \star G_j[1]$ is zero} \rc.
    \end{equation}
    We fix a maximal lift $\tilde{\varphi}_{i,j;b_{i,j}} \colon G_i \to T_{-b_{i,j}}G_j$.
    \smallskip

    \textbf{Step 3: Proof of the claim}. Given the alternative characterization of $b_{i,j}$, we can now prove the assertions in \cref{claim:general} as follows based on the induction hypothesis.

    (i) Consider the first assertion. First, by \eqref{eqn:definition_tildeb} and the induction hypothesis, we have
    \begin{equation}\label{eq:tildeb_set}
        \tilde{b}_{i,j} \in 
        \lc 
            \cA_{i,i_1}(p_{0,i_1}) + \cA_{i_1,i_2}(p_{i_1,i_2}) + \dots + \cA_{i_k,j}(p_{i_k,j})                
        \relmid
        \begin{aligned}
            & i < i_1 < \dots < i_k < j \ (k \ge 1)\\
            & p_{s,t} \in L_s \cap L_t
        \end{aligned}
        \rc
    \end{equation}
    Now we prove the assertion by splitting it into two cases.
    
    Suppose $b_{i,j}=\tilde{b}_{i,j}$. Then the first assertion of \cref{claim:general} follows by induction. 
    In this case, by \eqref{eq:bij_set} we find that the energy of ${\varphi_{i,j}}_*$ is
    \begin{equation}\label{eqn:energy_including_0}
         \min \lc \cA_{j}(p_{j}) - \cA_{i}(p_{i}) - \tilde{b}_{i,j} \ge 0 \relmid p_s \in N \cap L_s \ (s=i,j) \rc.
    \end{equation}
    
    Suppose $b_{i,j} < \tilde{b}_{i,j}$.
    In this case, we cannot extend a lift beyond $b_{i,j}$, which by \eqref{eqn:characterization-of-bij} means that $\Hom(G_i, \bfk_{M \times [-b_{i,j}-\varepsilon,0)} \star G_j) \to \Hom(G_i, \bfk_{M \times [-b_{i,j},0)} \star G_j)$ fails to be an isomorphism for $\varepsilon > 0$ sufficiently small. 
    By the exact triangle $\bfk_{M \times [-b,0)} \star G_j \to T_{-b}G_j \to G_j \toone$, this should happen only when $\Hom(G_i,T_{-b_{i,j}-\varepsilon}G_j) \to \Hom(G_i,T_{-b_{i,j}}G_j)$ fails to be an isomorphism. Note that
    \begin{equation*}
        \Gamma_{[a,+\infty)}(\pi_* \cHom^\star(G_0, G_1))_{a} = \Cone(\Hom(G_0, T_{a-\varepsilon}G_1) \to \Hom(G_0, T_{a}G_1)).
    \end{equation*}
    Hence by \cref{lem:ike-intersection} we get
    \begin{equation}\label{eq:b_set}
        b_{i,j} \in \lc \cA_{i,j}(p_{i,j}) \relmid p_{i,j} \in L_i \cap L_j \rc
    \end{equation}
    and the first assertion in \cref{claim:general} follows.
    In this case, by \eqref{eq:bij_set} we also find that the energy of ${\varphi_{i,j}}_*$ is at least
    \begin{equation}\label{eqn:energy_including_0_2}
         \min \lc \cA_{j}(p_{j}) - \cA_{i}(p_{i}) - b_{i,j} \ge 0 \relmid p_s \in N \cap L_s \ (s=i,j) \rc.
    \end{equation}

    (ii) Consider the second assertion. By the induction hypothesis, we know that the tower induced by the lifts
    \begin{align}
        \big[& \pi_*\cHom^\star(F_N, G_{i}) \to \pi_*\cHom^\star(F_N, T_{-b_{i,i+1}}G_{i+1}) \to \cdots \nonumber \\
        & \to \pi_*\cHom^\star(F_N, T_{-b_{i,j-2}}G_{j-2}) \to \pi_*\cHom^\star(F_N, T_{-b_{i,j-1}}G_{j-1}) \big] \label{eq:lift-of-complex}
    \end{align} 
    is a twisted complex formed by iterated mapping cones. By \eqref{eq:def-induction-b_ij}, there is a null homotopy $\tilde{\varphi}_{i,j; b_{i,j}}$ that lifts $\varphi_{i,j}$. Since $b_{i,t} - b_{i,s} \leq b_{s,t}$, by induction hypothesis and \eqref{eq:def-induction-b_ij}, there is a null homotopy $\tilde{\varphi}_{i,j; b_{i,t} - b_{i,s}}$ that lifts $\varphi_{s,t}$. Therefore, the tower induced by the lifts
    \begin{align*}
        \big[ & \pi_*\cHom^\star(F_N, G_{i}) \to \pi_*\cHom^\star(F_N, T_{-b_{i,i+1}}G_{i+1}) \to \cdots \nonumber \\
        & \to \pi_*\cHom^\star(F_N, T_{-b_{i,j-1}}G_{j-1}) \to \pi_*\cHom^\star(F_N, T_{-b_{i,j}}G_{j}) \big]
    \end{align*} 
    is still a twisted complex formed by iterated mapping cones. 
    This proves the second assertion of \cref{claim:general}.
    
    (iii) Consider the third assertion of \cref{claim:general}. By \eqref{eqn:energy_including_0} and \eqref{eqn:energy_including_0_2}, we only need to exclude the case of energy zero. By the definition of the energy, that means we only need to show that the complex \eqref{eq:lift-of-complex} induced by the lifts of chain maps does not introduce any cancellation of endpoints.
    Consider \cref{lem:ike-intersection} and let 
    \begin{equation*}
        \Lambda_N = \Lambda_N, \;\;
        \Lambda_L =\, \Lambda_{L_i} \cup T_{-b_{i,i+1}}(\Lambda_{L_{i+1}}) \cup \dots \cup T_{-b_{i,j}}(\Lambda_{L_{j}}).
    \end{equation*}
    Since $N \cap L_s \cap L_{t} = \varnothing$ for any $0 \leq s, t \leq n-1$, we know that $\Lambda_N$ and $\Lambda_L$ intersect cleanly, and the sheaf $[ G_i \to T_{-b_{i,i+1}}G_{i+1} \to \dots \to T_{-b_{i,j}}G_{j}  ]$ is simple along the smooth strata of $\Lambda_L$. Therefore, by \cref{lem:ike-intersection}, the number of end points of bars of 
    \begin{align*}
        \big[ & \pi_*\cHom^\star(F_N, G_i) \to \pi_*\cHom^\star(F_N, T_{-b_{i,i+1}}G_{i+1}) \to \cdots \nonumber \\
        & \to \pi_*\cHom^\star(F_N, T_{-b_{i,j-1}}G_{j-1}) \to \pi_*\cHom^\star(F_N, T_{-b_{i,j}}G_{j}) \big]
    \end{align*}
    is $\sum_{s=i}^{j}\sum_{n}\dim H^n(N \cap L_s; \bfk)$, and we conclude that there are no cancellations of endpoints.
    Hence, in each case, we find that the energy of ${\varphi_{i,j}}_*$ is at least
    \begin{equation*}
        \min \lc \cA_j(p_j) - \cA_i(p_i) - b_{i,j} > 0 \mid p_s \in N \cap L_s \ (s = i, j) \rc,
    \end{equation*}
    which proves the third assertion of \cref{claim:general}.   
\end{proof}

\begin{remark}
    In the proof, we fixed an identification (recall that $B(N\cap L_i)$ is a basis of $H_*(N \cap L_i; \bfk)$)
    \begin{equation*}
        \pi_*\cHom^\star(F_N,G_i) \simeq \lb \bigoplus\nolimits_{p \in B(N \cap L_i)} \bfk_{[\cA_i({p}),\infty)}[k_p], {d_i}_* \rb, 
    \end{equation*}
    which is not necessarily induced by a chain map. 
    Therefore, we need to fix chain maps that represent ${\varphi_{i,j}}_*$'s (and the morphisms induced by $\tilde{\varphi}_{i,j:b}$'s) by hand.
    
    More precisely, consider the dg-category $\cC$ formed by finite complexes of direct sums of a finite number of $\bfk_{[a,\infty)}$'s and the chain maps between them.
    The natural functor $i \colon \cC \to \cT(\pt)$ is fully faithful and hence equivalent to its essential image $\cD$, which contains $\pi_*\cHom^\star(F_N,G_i)$'s. There is no inverse dg-functor of the equivalence $\cC \to \cD$.
    
    However, there is an $A_\infty$-inverse functor $\cD \to \cC$. This inverse functor allows us to induce morphisms that represent ${\varphi_{i,j}}_*$'s in a functorial way. The higher operations of the $A_\infty$-operations will appear in components of~\eqref{eqn:twist-complex-null-homotopy}. This seems to be consistent with the description in \cite{BCS18}.
\end{remark}

\begin{remark}\label{rem:morse-bott-intersect}
    More generally, in \cref{theorem:intersection}, assuming all the pairwise intersections have finitely many connected components, we have
    \begin{equation*}
        \sum_{c\in \bR,n\in \bZ} \dim H^n(\Omega_+;\mu hom(F_N,T_c F_{L'})) \ge \sum_{c\in \bR,n\in \bZ} \sum_{i=0}^{m-1} \dim H^n(\Omega_+;\mu hom(F_N,T_c F_{L_i})),
    \end{equation*}
    where $\mu hom$ is a functor defined in \cite[Section~4.4]{KS90}.
    In the case of clean intersection, $H^n(\Omega_+;\mu hom(F_N,T_c F_{L_i}))$ computes the cohomology of $N \cap L_i$ (see \cite[Thm.~4.14]{Ike19}, for example). 
\end{remark}

\paragraph{Acknowledgements}

TA thanks Kaoru Ono for the helpful discussions.
TA was partially supported by Innovative Areas Discrete Geometric Analysis for Materials Design (Grant No.~JP17H06461).
YI thanks Alexandru Oancea for the fruitful discussions.
He also thanks Paul Biran, Octav Cornea, and Jun Zhang for helpful discussions and clarifications. 
YI was partially supported by JSPS KAKENHI Grant Numbers JP21K13801 and JP22H05107.
TA and YI were partially supported by JST, CREST Grant Number JPMJCR24Q1, Japan. 
WL thanks Xin Jin and David Treumann for helpful discussions and clarifications and Bingyu Zhang for helpful discussions.
The authors thank the anonymous referee for the helpful comments, which improved the exposition.

\printbibliography

% \clearpage

\noindent Tomohiro Asano: 
Research Institute for Mathematical Sciences, Kyoto University, \linebreak Kitashirakawa-Oiwake-Cho, Sakyo-ku, 606-8502, Kyoto, Japan.

\noindent \textit{E-mail address}: \texttt{tasano@kurims.kyoto-u.ac.jp}, \texttt{tomoh.asano@gmail.com}

\medskip

\noindent Yuichi Ike:
Institute of Mathematics for Industry, Kyushu University, 744 Motooka, Nishi-ku, Fukuoka-shi, Fukuoka 819-0395, Japan.

\noindent
\textit{E-mail address}: \texttt{ike@imi.kyushu-u.ac.jp}, \texttt{yuichi.ike.1990@gmail.com}

\medskip 

\noindent Wenyuan Li:
Department of Mathematics, University of Southern California, 3620 S Vermonth Ave, Los Angeles, CA 90089, United States.

\noindent
\textit{E-mail address}: \texttt{wenyuan.li@usc.edu}

\end{document}